\documentclass[11pt]{article}
\usepackage[margin=2.9cm]{geometry}
\RequirePackage[colorlinks,citecolor=blue,urlcolor=blue]{hyperref}

\usepackage{amssymb}
\usepackage{amsmath}
\usepackage{amsthm}
\usepackage{verbatim}
\usepackage{graphicx}
\usepackage{epstopdf}
\usepackage{amsmath}
\usepackage{bm}
\usepackage{setspace}
\usepackage[inline,shortlabels]{enumitem}
\usepackage{footnote}

\usepackage[font=small]{caption}

\usepackage[
	style=alphabetic,
	backend=bibtex,
	doi=false,
	isbn=false,
	url=false,
	firstinits=true,
	maxnames=4
	]{biblatex}			

\bibliography{sc_rev_refs}
\renewbibmacro{in:}{}			
\DeclareFieldFormat[article,inbook,incollection,inproceedings,patent,thesis,unpublished]{citetitle}{#1}
\DeclareFieldFormat[article,inbook,incollection,inproceedings,patent,thesis,unpublished]{title}{#1} 

\usepackage{algorithm}
\usepackage{algpseudocode}

\algnewcommand{\IIf}[1]{\State\algorithmicif\ #1\ \algorithmicthen}

\usepackage[matrixnorms,opnorm=op]{arash_macros}

\input{spec_defs}

\usepackage{widebar}

\title{Analysis of spectral clustering algorithms for community detection: the general bipartite setting}
\author{Zhixin Zhou and Arash A. Amini}

\begin{document}

\maketitle
\begin{abstract}
	We consider  spectral clustering algorithms for community detection under a general bipartite stochastic block model (SBM).
	A modern spectral clustering algorithm  consists of three steps: (1) regularization of an appropriate adjacency or Laplacian matrix (2) a form of spectral truncation and (3) a \kmeans type algorithm in the reduced spectral domain. We focus on the adjacency-based spectral clustering and for the first step,  propose a new data-driven regularization  that can restore the concentration of the adjacency matrix even for the sparse networks. This result is based on recent work on regularization of random binary matrices, but avoids using unknown population level parameters, and instead estimates the necessary quantities from the data. We also propose and study a novel variation of the spectral truncation step and show how this variation changes the nature of the misclassification rate in a general SBM. We then show how the consistency results can be extended to models beyond SBMs, such as inhomogeneous random graph models with approximate clusters, including a  graphon clustering problem, as well as general sub-Gaussian biclustering. A theme of the paper is providing a better understanding of the analysis of spectral methods for community detection and establishing consistency results, under fairly general clustering models and for a wide regime of degree growths, including sparse cases where the average expected degree grows arbitrarily slowly.
	
	
	
\end{abstract}

\section{Introduction}\label{sec:intro}

Spectral clustering is one of the most popular and successful approaches to clustering and has appeared in various contexts in statistics and computer science among other disciplines. The idea generally applies when one can define a similarity matrix between pairs of objects to be clustered~\cite{ng2002spectral,von2007tutorial,von2008consistency}. Recently,  there has been a resurgence of interest in the spectral approaches and their analysis in the context of network clustering
~\cite{boppana1987eigenvalues,McSherry,dasgupta2006spectral,coja2010graph,Tomozei2011,rohe2011spectral,balakrishnan2011noise,chaudhuri2012spectral,fishkind2013consistent,qin2013regularized,Joseph2013,Krzakala2013,lei2015consistency,yun2014accurate,yun2014community,Binkiewicz2014, lyzinski2014perfect, chin2015stochastic,gao2017achieving}. 
 The interest has been partly fueled by the recent  activity in understanding statistical network models for clustering and community detection, and in particular by the flurry of theoretical work on the stochastic block model (SBM)---also called the planted partition model---and its variants~\cite{abbe2017community}.
 
  There has been  significant recent  advances in analyzing spectral clustering approaches in SBMs. 
  We start by identifying the main components of the analysis in Section~\ref{sec:analysis:steps}, and then show how variations  can be introduced at each step to obtain improved consistency results and novel algorithms. We will work in the general bipartite setting which has received comparatively less attention, but most results in the paper can be easily extended to the (symmetric) unipartite models (cf. Remark~\ref{rem:sym:case}).  It is worth noting that clustering bipartite SBMs is closely related to the biclustering~\cite{hartigan1972direct} and  co-clustering~\cite{dhillon2001co,rohe2012co} problems. 
  
 
 A modern  spectral clustering algorithm often consists of three steps:
 \begin{enumerate*}[label=(\alph*)]
 	\item the regularization and concentration of the adjacency matrix (or the Laplacian)
 	\item the spectral truncation step, and
 	\item the \kmeans step.
 \end{enumerate*}
 By using variations in each step one obtains different spectral algorithms, which is then reflected in the variations in the consistency results. 
 
 The regularization step is fairly recent and is motivated by the observation that proper regularization significantly improves the performance of  spectral methods in sparse networks~\cite{chaudhuri2012spectral,amini2013pseudo,Joseph2013,Le2015,chin2015stochastic}. In particular, regularization restores the concentration of the adjacency matrix (or the Laplacian) around its expectation in the 
 \emph{sparse regime}, where the average degree of the network is constant or grows very slowly with the  number of nodes. 
 In this paper, building on these recent advances, we introduce a novel regularization scheme for the adjacency matrix that is fully data-driven and avoids relying on unknown quantities such as maximum expected degrees (for rows and columns) of the network. This regularization scheme is introduced as Algorithm~\ref{alg:deg:reg} in Section~\ref{sec:concent} and we show that under a general SBM it achieves the same concentration bound (Theorem~\ref{thm:concent:data:driven}) as its oracle counterpart.

 For the spectral truncation step,  we will consider three variations, one of which (Algorithm~\ref{alg:scone}) is the common approach of keeping the top $k$ leading eigenvectors as columns of the matrix passed to the \kmeans step. In this traditional approach, the spectral truncation can be viewed as obtaining a low-dimensional representation of the data, suitable for an application of  simple \kmeans type algorithms. We also consider a recent variant (Algorithm~\ref{alg:scrr}) proposed in~\cite{yun2014community,gao2017achieving}, in which the spectral truncation step acts more as a denoising step. We then propose a third alternative (Algorithm~\ref{alg:scerr} in Section~\ref{sec:consist:res}) which combines the benefits of both approaches while improving the computational complexity. One of our novel contributions is to derive consistency results for Algorithms~\ref{alg:scrr} and~\ref{alg:scerr}, under a general SBM, showing that the behavior of the two algorithms is the same (but different than Algorithm~\ref{alg:scone}) assuming that the \kmeans step satisfies a property we refer to as \emph{isometry invariance} (Theorems~\ref{thm:scrr} and~\ref{thm:scerr}).  
 
  In the final step of spectral clustering, one runs a \kmeans type algorithm on the output of the truncation step. We discuss this step in some detail since it is often  mentioned briefly in the analyses of spectral clustering, with the exception of a few recent works~\cite{lei2015consistency,gao2017achieving, gao2018community}. By the \kmeans step, we do not necessarily mean the solution of the well-known \kmeans problem, although, this step is usually implemented by an algorithm that approximately minimizes the \kmeans objective. We will consider the \kmeans step in some generality, by introducing the notation of a \emph{\kmeans matrix} (cf. Section~\ref{sec:kmeans:step}). The goal of the final step of spectral clustering is to obtain a \kmeans matrix which is close to the output of the truncation step. We characterize sufficient conditions on this approximation so that the overall algorithm produces consistent clustering. Any approach that satisfies these conditions can be used in the \kmeans step, even if it is not an approximate \kmeans solver. 
  
  Most of the above ideas extend beyond SBMs and, in Section~\ref{sec:exten}, we consider various extensions. We first consider some unique aspects of the bipartite setting, for example, the possibility of having clusters only on one side of the network, or having more clusters on one side than the rank of the expected adjacency matrix. We then show how the results extend to general sub-Gaussian biclustering (Section~\ref{sec:gen:subg}) and general inhomogeneous random graphs with approximate clusters (Section~\ref{sec:inhom:graphs}).
  
  
  The organization of the rest of the paper is as follows: After introducing  some notation in Section~\ref{sec:notation}, we discuss the general (bipartite) SBM in Section~\ref{sec:SBM}. An outline of the analysis is given in Section~\ref{sec:analysis:steps}, which also provides the high-level intuition of why the spectral clustering works in SBMs and what each step will achieve in terms of guarantees on its output. 
  This section provides a typical blueprint theorem on consistency (Theorem~\ref{thm:sc:prototype} in Section~\ref{sec:analysis:sketch}) which serves as a prelude to later consistency results. The rest of Section~\ref{sec:analysis:steps} provides the details of the last two steps of the analysis sketched in Section~\ref{sec:analysis:sketch}. The regularization and concentration (the first step) is detailed in Section~\ref{sec:concent} where we also introduce our data-driven regularization.
  We then give explicit algorithms and their corresponding consistency results in Section~\ref{sec:consist:res}. Extensions of the results are discussed in Section~\ref{sec:exten}. Some simulations  showing the effectiveness of the regularization are provided in Section~\ref{sec:sim}.

 	\paragraph{Related work.} There are numerous papers discussing aspect of spectral clustering and its analysis. 
 	Our paper is mostly inspired by recent developments in the field, especially by the consistency results of~\cite{lei2015consistency,yun2014accurate,chin2015stochastic,gao2017achieving} and concentration results for the regularized adjacency (and Laplacian) matrices such as~\cite{Le2015,bandeira2016sharp}. Theorem~\ref{thm:scone} on the consistency of the typical adjacency-based spectral clustering algorithm---which we will call \scone---is generally known~\cite{lei2015consistency,gao2017achieving}, though our version is slightly more general; see Remark~\ref{rem:cond:relax}. The spectral algorithm \scrr is proposed and analyzed in~\cite{yun2014community,gao2017achieving,gao2018community} for special cases of the SBM (and its degree-corrected version); the new consistency result we give for \scrr is for the general SBM and reveals the contrast with \scone. Previous analyses did not reveal this difference due to focusing on the special case; see Examples~\ref{exa:sbpp} and~\ref{exa:diff} in Section~\ref{sec:consist:res} for details.  We also propose the new \scerr which has the same performance as \scrr (assuming a proper \kmeans step) but is much more computationally efficient. The results that we prove for \scrr and \scerr can be recast in terms of the \emph{mean parameters} of the SBM (in contrast to \scone), as demonstrated in Corollary~\ref{cor:mean:param} of Section~\ref{sec:res:mean:param}. For an application of this result, we refer to our work on optimal bipartite  network clustering~\cite{pl-bipartite}.

\subsection{Notation} \label{sec:notation}
\paragraph{Orthogonal matrices.}
We write $\ort{n}k$ for the set of $n \times k$ matrices with orthonormal columns. The condition $k \le n$ is implicit in defining $\ort{n}k$. The case $\ort{n}n$ is the set of orthogonal matrices, though with some abuse of terminology we also refer to matrices in $\ort{n}k$ as orthogonal even if $k < n$. Thus, $Z \in \ort{n}{k}$ iff $Z^T Z = I_k$.
We also note that $Z \in \ort{n}{k_1}$ and $U \in \ort{k_1}{k}$ implies $Z U \in \ort{n}k$. The following holds:
\begin{align}\label{eq:ort:isom:1}
	\norm{U x}_2 = \norm{x}_2, \quad \forall x \in \reals^k, \; U \in \ort{k_1}{k}, 
\end{align}
for any $k_1 \ge k$. On the other hand, 
\begin{align}\label{eq:ort:isom:2}
	\norm{ U^T x}_2 \le  \norm{x}_2, \quad \forall x \in \reals^{k_1}, \; U \in \ort{k_1}{k},
\end{align}
where equality holds for all $x \in \reals^{k_1}$, iff $k_1 = k$. To see~\eqref{eq:ort:isom:2}, let $u_1,\dots,u_{k} \in \reals^{k_1}$ be the columns of $U$, constituting an orthonormal sequence which can be completed to an orthonormal basis by adding say $u_{k+1},\dots,u_{k_1}$. Then, $\norm{ U^T x}_2^2 = \sum_{j=1}^{k} \ip{u_j,x}^2 \le \sum_{j=1}^{k_1} \ip{u_j,x}^2 = \norm{x}_2^2$.

\paragraph{Membership matrices and misclassification.} We let $\hr(n,\kk)$ denote the set of \emph{hard} cluster labels: $\{0,1\}$-valued $n \times \kk$ matrices where each row has exactly a single 1. A matrix $Z \in \hr(n,\kk)$ is also called a membership matrix, where row $i$ is interpreted as the membership of node $i$ to one of $k$ clusters (or communities). Here we implicitly assume that we have a network on nodes in $[n] =\{1,\dots,n\}$, and there is a latent partition of $[n]$ into $k$ clusters. In this sense, $Z_{ik} = 1$ iff node $i$ belongs to cluster $k$. Given, two membership matrices $Z,Z' \in \hr(n,\kk)$, we can consider the average misclassification rate between them, which we denote as $\Misb(Z,Z')$: Letting $z_i^T$ and $(z_i')^T$ denote the $i$th row of $Z$ and $Z'$ respectively, we have
\begin{align}\label{eq:miss:avg}
	\Misb(Z,Z') := \min_Q  \frac1n \sum_{i=1}^n 1\{z_i \neq Q z'_i\}
\end{align}
where the minimum is taken over $k\times k$ permutations matrices $Q$. We also let $\Mis_r(Z,Z')$ be the misclassification rate between the two, over the $r$th cluster of $Z$, that is, $\Mis_r(Z,Z') =  \frac1{n_r} \sum_{i:\; z_i = r} 1\{z_i \neq Q^* z'_i\}$ where $n_r = \sum_{i=1}^n 1\{z_i = r\}$ is the size of the $r$th cluster of $Z$, and $Q^*$ is the optimal permutation in~\eqref{eq:miss:avg}. Note that in contrast to $\Misb$, $\Mis_r$ is not symmetric in its two arguments. 
We also write $\Mis_\infty := \max_r \Mis_r$. These definitions can be extended to misclassification rates between \kmeans matrices introduced in Section~\ref{sec:kmeans:step}.



\section{Stochastic Block Model}\label{sec:SBM}
We consider the general, not necessarily symmetric, Stochastic Block Model (SBM) with bi-adjacency matrix $A \in \{0,1\}^{n_1 \times n_2}$. We assume throughout that $n_2 \ge n_1$, without loss of generality.  We have membership matrices $Z_r \in \hr(n_r,\kk_r)$ for each of the two sides $r=1,2$, where $\kk_r \le n_r$ denotes the number of communities on side $r$. Each element of $A$ is an independent draw from a Bernoulli variable, and
\begin{align}\label{eq:nonsym:mean:def}
P := \ex[A] = Z_1 B Z_2^T, \quad B = \frac{\Psi}{\sqrt{n_1 n_2}}
\end{align}
where $B \in [0,1]^{\kk_1\times \kk_2}$ is the connectivity--- or the edge probability---matrix, and $\Psi$ is a rescaled version. We also use the notation
\begin{align}\label{eq:matrix:ber:defn}
	A \sim \ber(P) \;\iff \;A_{ij} \sim \ber(P_{ij}), \; \text{independent across  $(i,j) \in [n_1]\times [n_2]$.}
\end{align}
Classical SBM which we refer to as \emph{symmetric SBM}  corresponds to the following modifications:
\begin{enumerate}[label=(\alph*)]
	\item $A$ is assumed to be symmetric: Only the upper diagonal elements are drawn independently and the bottom half is filled symmetrically. For simplicity, we allow for self-loops, i.e., draw the diagonal elements from the same model. This will have a negligible effect in the arguments.
	
	\item $n_1 = n_2 =n, \; \kk_1 = \kk_2 = \kk, \; Z_1 = Z_2 = Z$.
	
	\item $B$ is assumed symmetric. 
\end{enumerate}
We note that~\eqref{eq:nonsym:mean:def} still holds over all the elements.
Directed SBM is also a special case, where~(b) is assumed but not~(a) or~(c). That is,  $A$ is not assumed to be symmetric and all the entries are independently drawn, while $B$  may or may not be symmetric.
%

\medskip
We refer to $P$ as the \emph{mean matrix}  and note that it is of rank at most $\kk := \min\{\kk_1, \kk_2\}$. 
Often $k \ll n_1,n_2$, that is $P$ is a low-rank matrix which is the key in why spectral clustering works well for SBMs. However, the case where either $k_1 \gtrsim n_1$ or $k_2 \gtrsim n_2$ is allowed. (Here, $k_1 \gtrsim n_1$ means $k_1 \ge c n_1$ for some universal constant $c > 0$.) An extreme example of such setup can be found in Section~\ref{sec:one:cluster}.

We let $N_r = \diag(n_{r1},\dots,n_{rk_r})$ for $r=1,2$ where $n_{rj}$ is the size of the $j$th cluster of $Z_r$; that is, $N_r$ is a diagonal matrix whose diagonal elements are the sizes of the clusters on side $r$. We also consider the normalized version of $N_r$, 
\begin{align}\label{eq:clust:prop:def}
\Nb_r := N_r / n_r = \diag(\pi_{r1},\dots,\pi_{r\kk_r}), 
\end{align}
collecting the cluster proportions $ \pi_{rj} := n_{rj}/n_r$.
Let us define 
\begin{align*}
\Bb := N_1^{1/2} B N_2^{1/2} = \Nb_1^{1/2} \Psi \Nb_{2}^{1/2}.
\end{align*}
For sparse graphs, we expect $\Nb_r$ and $\Psi$ to remain stable as $n_r \to \infty$, hence $\Bb$ remains stable; in contrast, the entries of $B$ itself vanish under scaling~\eqref{eq:nonsym:mean:def}. See Remark~\ref{rem:scaling} below for details. Throughout the paper, barred parameters refer to quantities that remain stable or slowly vary with $n_r \to \infty$. The following lemma is key in understanding spectral clustering for SBMs:
\begin{lem}\label{lem:P:svd}
	Assume that $\Bb = \Usi \Sigma \Vsi^T$ is the reduced SVD of $\Bb$, where $\Usi \in \ort{\kk_1}{\kk}, \Vsi \in \ort{\kk_2}{\kk}$, $\Sigma = \diag(\sigma_1,\dots,\sigma_\kk)$, and $\kk = \min\{\kk_1,\kk_2\}$. Then,
	\begin{align}\label{eq:P:svd}
	P = (\Zb_1 \Usi) \,\Sigma\, (\Zb_2 \Vsi)^T
	\end{align}
	is the reduced SVD of $P$ where $\Zb_r := Z_r N_r^{-1/2}$  is itself an orthogonal matrix, $\Zb_r^T \Zb_r = I_{\kk_r}$.
\end{lem}
\begin{proof}
	We first show that $Z_r^T Z_r = N_r$ which then implies that $\Zb_r$ is orthogonal. Let $z_{ri}^T$ be the $i$th row of $Z_r$ and note that $Z_r^T Z_r = \sum_{i=1}^n z_{ri}^{} z_{ri}^T$. Since $z_{ri} \in \hr(1,\kk_r)$, each $z_{ri}^{} z_{ri}^T$ is a diagonal matrix with a single $1$ on the diagonal at the position determined by the cluster assignment of node $i$ on side $r$.
	Now, a little algebra on~\eqref{eq:nonsym:mean:def} shows that $P =  \Zb_1 \Bb \Zb_2^T$, hence~\eqref{eq:P:svd} holds. Since $\Zb_r \in \ort{n_r}{\kk_r}$, we have $\Zb_1 \Usi \in \ort{n_1}{k}$ and $\Zb_2 \Vsi \in \ort{n_2}{k}$ showing that~\eqref{eq:P:svd} is in fact a (reduced) SVD of~$P$.
\end{proof}

\smallskip

When dealing with the symmetric SBM, we will drop the subscript $r$ from all the relevant quantities; for example, we write $N = N_1 = N_2$, $\Zb = \Zb_1 = \Zb_2$, $\pi_{j} = \pi_{1j} = \pi_{2j}$, and so on.

\begin{rem}[Reduced versus truncated]\label{rem:reduce:trunc}
	 The term \emph{reduced SVD} in Lemma~\ref{lem:P:svd} (also known as compact SVD) means that we reduce the orthogonal matrices in a full SVD by removing the columns corresponding to zero singular values. The number of columns of the resulting matrices will be equal to the rank of the underlying matrix (i.e., both $\Zb_1 \Usi$ and $\Zb_2 \Vsi$ will have $k = \min\{k_1,k_2\}$ columns in the case of $P$). Hence, a reduced SVD is still an exact SVD. Later, we will use the term \emph{truncated SVD} to refer to an ``approximation'' of the original matrix by a lower rank matrix obtained by further removing columns corresponding to small nonzero singular values (starting from a reduced SVD).  Hence, a truncated SVD is only an approximation of the original matrix.
\end{rem}

\begin{rem}[Scaling and sparsity]\label{rem:scaling}
	As can be seen from the above discussion,
	the normalization in~\eqref{eq:nonsym:mean:def} is natural for studying spectral clustering. In the symmetric case, where $n_1 = n_2 = n$, the normalization reduces to $B = \Psi/n$, which is often assumed when studying sparse SBMs by requiring that either $\infnorm{\Psi}$ is $O(1)$ or grows slowly with $n$. To see why this implies a sparse network, note that the expected average degree of the symmetric SBM (under this scaling) is
	\begin{align*}
		\frac1n 1^T P 1 = \frac1n 1^T N B N 1 =  1^T \Nb \Psi  \Nb 1 = \sum_{i,j=1}^\kk \pi_{i} \pi_{j} \Psi_{ij} =: \dav
	\end{align*}
	using $1_n^T Z = 1_\kk^T N$. (Here and elsewhere, $1$ is the vector of all ones of an appropriate dimension; we write $1_n$ if we want to emphasize the dimension $n$.) Thus, the growth of the average expected degree, $\dav$, is the same as $\Psi$, and as long as 
	$\Psi$ is $O(1)$ or grows very slowly with $n$, the network is sparse. Alternatively, we can view the expected density of the network (the expected number of edges divided by the total number of possible edges) as a measure of sparsity. For the symmetric case, the expected density is $(\frac12 1^T P 1) / \binom{n}2 \sim \dav/n$ and is $O(n^{-1})$ if $\dav=O(1)$. Similar observations hold in the general bipartite case if we let $ n =\sqrt{n_1 n_2}$, the geometric mean of the dimensions. The expected density of the bipartite network under the scaling of~\eqref{eq:nonsym:mean:def} is 
	\begin{align*}
		\frac{1^T P 1}{n^2} = \frac1{n^2} 1^T N_1 B N_2 1 = \frac1n 1^T \Nb_1 \Psi \Nb_2 1 = \frac{\dav}{n}, \quad (n = \sqrt{n_1n_2})
	\end{align*} 
	where $\dav := 1^T \Nb_1 \Psi \Nb_2 1 = \sum_{i,j} \pi_{1i}\pi_{2j} \Psi_{ij}$ can be thought of as the analog of the expected average degree in the bipartite case. As long as $\infnorm{\Psi}$ grows slowly relative to $n = \sqrt{n_1n_2}$, the bipartite network is sparse.
\end{rem}

\section{Analysis steps}\label{sec:analysis:steps}
Throughout, we focus on recovering the row clusters. Everything that we discuss goes through, with obvious modifications, for recovering the column clusters. Recalling the decomposition~\eqref{eq:P:svd}, the idea of spectral clustering in the context of SBMs is that $\Zb_1 \Usi$ has enough information for recovering the clusters and can be obtained by computing a reduced SVD of $P$. In particular, applying a \kmeans type clustering on the rows of $\Zb_1 \Usi$ should recover the cluster labels. On the other hand, the actual random adjacency matrix, $A$, is concentrated around the mean matrix $P$, after proper regularization if need be. We denote this potentially \emph{regularized version as $\Are$}. Then, by  the spectral perturbation theory, if we compute a reduced SVD of $\Are = \Zh_1 \Sigh \Zh_2^T$ where $\Zh_r \in \ort{n_r}{k}$, $r=1,2$ and $\Sigh$ is diagonal, we can conclude that $\Zh_1$ concentrates around $\Zb_1 \Usi$. Hence, applying a continuous \kmeans  algorithm on $\Zh_1$ should be able to recover the labels with a small error.

\subsection{Analysis sketch}\label{sec:analysis:sketch}
Let us sketch the argument above in more details. A typical approach in proving consistency of spectral clustering consists of the following steps:
\begin{enumerate}[label=\textbf{\arabic*.},wide]
	\item We replace $A$ with a properly regularized version $\Are$. We provide the details for one such regularization in Theorem~\ref{thm:concent:nonsym} (Section~\ref{sec:concent}). However, the only property we require of the regularized version is that it concentrates, with high probability, around the mean of $A$, at the following rate (assuming $n_2 \ge n_1$):
	\begin{align}\label{eq:gen:concent}
	\opnorm{\Are - \ex [A]} \le C \sqrt{d}, 
	\quad \text{where} \quad \dg = \sqrt{\frac{n_2}{n_1}}\norm{\Psi}_\infty.
	\end{align}
	Here and throughout $\opnorm{\cdot}$ is the $\ell_2 \to \ell_2$ operator norm and $\infnorm{\Psi} = \max_{ij} \Psi_{ij}$.
	
	\item We pass from $\Are$ and $P = \ex[A]$ to their (symmetrically) dilated versions $\Are^\dagger$ and $P^\dagger$. The symmetric dilation operator will be given in~\eqref{eq:dilation:defn} (Section~\ref{sec:dilation}) and allows us to use spectral perturbation bounds for symmetric matrices. A typical final result of this step is
	\begin{align}\label{eq:Z:dev:bound}
	\fnorm{\Zh_1 - \Zb_1 \Usi Q} \le \frac{C_2}{\sigma_k} \sqrt{k \dg}, \quad \text{w.h.p.}
	\end{align}
	for some $Q \in \ort{k}{k}$. We recall that $\fnorm{\cdot}$ is the Frobenius norm. Here, $\sigma_k$ is the smallest nonzero singular value of $\Bb$ as defined in Lemma~\ref{lem:P:svd}. The form of~\eqref{eq:Z:dev:bound} will be different if instead of $\Zh_1$ one considers other objects as the end result of this step; 
	 see Section~\ref{sec:consist:res} (e.g.,~\eqref{eq:RR:dev}) for instances of such variations. The appearance of $Q$ is inevitable and is a consequence of the necessity of \emph{properly aligning} the bases of spectral subspaces, before they can be compared in Frobenius norm (cf.~Lemma~\ref{lem:DK:Z:dev}). Nevertheless, the growing stack of orthogonal matrices on the RHS of $\Zb_1$ has little effect on the performance of row-wise \kmeans, as we discuss shortly.
	
	\item The final step is to analyze the effect of applying a \kmeans algorithm to $\Zh_1$. Here, we introduce the concept of a \emph{\kmeans matrix}, one whose rows take at most $k$ distinct values. (See Section~\ref{sec:kmeans:step} for details). A \kmeans algorithm $\kalg$ takes a matrix $\Xh \in \reals^{n \times \md}$ and outputs a \kmeans matrix $\kalg(\Xh) \in \reals^{n \times \md}$. Our focus will be on \kmeans algorithms with the following property: If $\Xs \in \reals^{n \times \md}$ 
	is a \kmeans matrix, then for some constant $c > 0$,
	\begin{align}\label{eq:lqc:kmeans}
	\fnorm{\Xh - \Xs}^2 \le \eps^2 \;\implies\; \Misb(\kalg(\Xh), \Xs) \le c\, \eps^2/(n\delta^2).
	\end{align}
	where $\delta^2 = \delta^2(\Xs)$ 
	is the minimum center separation of $\Xs$ (cf. Definition~\ref{dfn:center:sep}), and 
	$\Misb$ is the average misclassification rate between two \kmeans matrices. 
	For future reference, we refer to property~\eqref{eq:lqc:kmeans} as the \emph{local quadratic continuity (LQC)} of algorithm $\kalg$; see Remark~\ref{rem:LQC} for the rationale behind the naming.
	As will become clear in Section~\ref{sec:kmeans:step}, \kmeans matrices encode both the cluster label information and cluster center information, and these two pieces can be recovered from them in a lossless fashion. Thus, it makes sense to talk about misclassification rate between \kmeans matrices, by interpreting it as a statement about their underlying label information. In Section~\ref{sec:kmeans:step}, we will discuss \kmeans algorithms that satisfy~\eqref{eq:lqc:kmeans}. 
	
	
\end{enumerate}

The preceding three steps of the analysis follow the three steps of a general spectral clustering algorithm, which we refer to as \emph{regularization}, \emph{spectral truncation} and \emph{\kmeans} steps, respectively.
Recalling the definition of cluster proportions, let us assume for some $\beta_r \ge 1$,
\begin{align}\label{asu:clust:prop}
	\max_{(t,s):\; t\neq s} \frac{2}{\pi_{rt}^{-1} + \pi_{rs}^{-1}}\; \le\;  \frac{\beta_r}{\kk_r}, \quad r=1,2. \tag{A1}
\end{align}
The LHS is the maximum \emph{harmonic mean} of pairs of distinct cluster proportions. For balanced clusters, we have $\pi_{rt} = 1/\kk_r$ for all $t \in [\kk_r]$ and we can take $\beta_r=1$. In general, $\beta_r$ measures the deviation of the clusters (on side $r$) from balancedness.
The following is a prototypical consistency theorem for a spectral clustering algorithm:
\begin{thm}[Prototype SC consistency]\label{thm:sc:prototype}
	Consider a spectral algorithm with a \kmeans step satisfying~\eqref{eq:lqc:kmeans}, and the ``usual'' spectral truncation step, applied to a regularized bi-adjacency matrix $\Are$ satisfying concentration bound~\eqref{eq:gen:concent}. Let $\kalg(\Zh_1)$ be the resulting estimate for membership matrix $Z_1$, and assume $\kk_1 = \kk =: \min\{\kk_1,\kk_2\}$. Then, under the SBM model of Section~\ref{sec:SBM} and assuming~\eqref{asu:clust:prop}, w.h.p.,
	\begin{align*}
		\Misb(\kalg(\Zh_1),\Zb_1) \, \lesssim \, \beta_1 \Big(\frac{\dg}{\sigma_\kk^2}\Big).
	\end{align*}
\end{thm}
Here, and in the sequel, ``with high probability'', abbreviated w.h.p., means with probability at least $1-n^{-c_1}$ for some universal constant $c_1 > 0$. The notation $ f \lesssim g$ means $ f \le c_2\, g$ where $c_2 > 0$ is a universal constant. In addition, $f \asymp g$ means $f \lesssim g$ and $g \lesssim f$.

\begin{proof}
	By assumption, concentration bound~\eqref{eq:gen:concent} holds. By Lemma~\ref{lem:DK:Z:dev} in Section~\ref{sec:dilation}, \eqref{eq:gen:concent} implies~\eqref{eq:Z:dev:bound} for the usual truncation step.
	Let $O := \Usi Q$ and $\eps^2 =C_2^2\, k\dg /\sigma_k^2$ so that~\eqref{eq:Z:dev:bound} reads
	$
	\fnorm{\Zh_1 - \Zb_1 O} \le \eps^2.
	$
	The \kmeans step satisfies~\eqref{eq:lqc:kmeans} by assumption. 
	Applying~\eqref{eq:lqc:kmeans} with $\Xh = \Zh_1$, $\Xs = \Zb_1 O$ and $\eps^2$ defined earlier leads to 
	\begin{align}\label{eq:mis:eps:delta:1}
		\Misb(\kalg(\Zh_1),\Zb_1) = \Misb(\kalg(\Zh_1),\Zb_1 O)
		\; \lesssim \;\frac{\eps^2}{ n_1 \delta^2}. 
	\end{align}
	
	It remains to calculate the minimum center separation of $\Xs =  \Zb_1 O \in \rmat(n_1,\kk)$, where $\Zb_1 \in \ort{n_1}{\kk_1}$ and $O := \Usi Q \in \ort{\kk_1}{\kk}$. We have
	\begin{align*}
		\delta^2 = \delta^2(\Zb_1 O) = \delta^2(\Zb_1)
		 &= \min_{t \neq s} \norm{n_{1t}^{-1/2} e_{t} - n_{1s}^{-1/2} e_{s}}_2^2 = \min_{t \neq s} \; (n_{1t}^{-1} + n_{1s}^{-1})
	\end{align*}
	where $e_s \in \reals^{\kk_1}$ is the $s$th standard basis vector.  The second equality uses invariance of $\delta^2$ to right-multiplication by a square orthogonal matrix. This is a consequence of $\norm{u^T O - v^T O}_2 = \norm{u - v}_2$ for $u,v \in \reals^{\kk_1}$ and $O \in \ort{\kk_1}{\kk}$ when $\kk_1 = \kk$; see~\eqref{eq:ort:isom:1}. The third equality is from the definition $\Zb_1 = Z_1 N_1^{-1/2}$. Using~\eqref{asu:clust:prop},
	\begin{align*}
		(n_1 \delta^2 )^{-1} \le \max_{t\neq s} \;(\pi_{1t}^{-1} + \pi_{1s}^{-1})^{-1} \le \frac{\beta_1}{2 \kk_1}.
	\end{align*}
	It follows that
		$
		\frac{\eps^2}{ n_1 \delta^2}\lesssim \beta_1 \frac{\kk}{\kk_1} \frac{\dg}{\sigma_\kk^2}$
	which gives the result in light of~\eqref{eq:mis:eps:delta:1} and assumption $\kk_1 = \kk$.
\end{proof}
The case $\kk_1 > \kk$ will be discussed in Section~\ref{sec:more:clust:than:rank}. For~\eqref{eq:lqc:kmeans} to hold for a \kmeans algorithm, one usually requires some additional constraints on $\eps^2/(n\delta^2)$, ensuring  that this quantity is small. We will restate Theorem~\ref{thm:sc:prototype} with such conditions explicitly once we consider the details of some \kmeans algorithms. For now, Theorem~\ref{thm:sc:prototype} should be thought of as a general blueprint, with specific variations obtained in Section~\ref{sec:consist:res} for various spectral clustering algorithms.

\begin{rem}\label{rem:why:it:is:consistency}
	To see that Theorem~\ref{thm:sc:prototype}  is a consistency result, consider the typical case where $\beta_1 \asymp 1$,  and $\sigma_\kk \asymp \dg$, so that $	\Misb(\kalg(\Zh_1),\Zb_1) = O(\dg^{-1})$. Then, as long as $\dg \to \infty$, i.e., the average degree of the network grows with $n$, assuming $n_1 \asymp n_2 \asymp n$ (for some $n$), we have $\Misb(\kalg(\Zh_1),\Zb_1) = o(1)$, i.e., the average misclassification rate vanishes with high probability. One might ask why $\sigma_{k} \asymp \dg$ is reasonable. Consider the typical case where $\Psi = \rho_n \Psi'$ for some constant matrix $\Psi'$  and a scalar parameter $\rho_n$ that captures the dependence on $n$. This setup is common in network modeling~\cite{bickel2009nonparametric}. Then, $d \asymp \infnorm{\Psi} = \rho_n \infnorm{\Psi'}$ and $\sigma_k = \rho_n \sigma'_k$ where $\sigma'_k$ is defined based on $\Psi'$ and hence constant. It follows that $d \asymp \rho_n \asymp \sigma_k$ and $\Misb(\kalg(\Zh_1),\Zb_1) = o(\rho_n^{-1})$. Note that, in general, $\rho_n$ can grow as fast as $n$ (cf.~\eqref{eq:nonsym:mean:def}).
	 More specific examples are given in Section~\ref{sec:consist:res}. 
\end{rem}

\begin{rem}\label{rem:cond:relax}
	Condition~\eqref{asu:clust:prop} is more relaxed that what is commonly assumed in the literature (though the proof is the same). Stating the condition as a harmonic mean allows one to have similar results as the balanced case when one cluster is large, while others remain more or less balanced. For example, let $\pi_{r1} = 1-c$ for some constant $c \in (0,1)$, say $c = 0.4$, and let $\pi_{rt} = c/(\kk_r-1)$ for $t\neq 1$. Then, we have for $s\neq t$
	\begin{align*}
		 \frac{2}{\pi_{rt}^{-1} + \pi_{rs}^{-1}} \le 2\min\{\pi_{rt},\pi_{rs}\} = \frac{2c}{\kk_r-1} \le \frac{4}{\kk_r}
	\end{align*}
	assuming $k_r \ge 2$. Hence~\eqref{asu:clust:prop} holds with $\beta_r = 4$. Note that as $k_r$ is increased, all but one cluster get smaller.
%
\end{rem}

\begin{rem}[LQC naming]\label{rem:LQC}
	The rationale behind the naming of~\eqref{eq:lqc:kmeans} is as follows: Let $\Delta(X,Y) := \frac1{\sqrt{n}} \fnorm{X - Y}$ be the metric induced by the normalized Frobenius norm on the space of $n \times m$ matrices. Assume that $X^*$ is a \kmeans matrix and that \kmeans matrices are fixed points of algorithm $\kalg$, hence $X^*= \kalg(X^*)$. Then, by taking the infimum over $\eps^2$, \eqref{eq:lqc:kmeans} can be written as
	\begin{align}\label{eq:kmeans:modulus}
		\Misb\big(\kalg(\Xh), \kalg(X^*)\big) \le \omega \big(\Delta(\Xh,\Xs)\big)
	\end{align}
	where $\omega(t) = c t^2 / \delta^2(X^*)$, showing that $\kalg$ is continuous w.r.t.  the two (pseudo)-metrics, locally at $X^*$, with a quadratic modulus of continuity. Note that this continuity is only required to hold around a \kmeans matrix $X^*$ (and not in general). Another reason for the ``locality'' is that such statements often only hold for sufficiently small $\Delta(\Xh,\Xs)$.
\end{rem}
In this rest of this section, we fill in some details of the last two steps of the plan sketched above, deferring Step~1 to Section~\ref{sec:concent}. 

\subsection{SV truncation (Step~2)}\label{sec:dilation}
We now show how the concentration bound~\eqref{eq:gen:concent} implies the deviation bound~\eqref{eq:Z:dev:bound} for the usual spectral truncation step.
Let us define the \emph{symmetric dilation} operator $: \reals^{n_1 \times n_2} \to \sym^{n_1 + n_2}$ by
\begin{align}\label{eq:dilation:defn}
\Pdag := 
\begin{pmatrix}
0 & P \\
P^T & 0
\end{pmatrix},
\end{align}
where $\sym^{n}$ is the set of symmetric $n\times n$ matrices. This operator will be very useful in translating the results between the symmetric and non-symmetric cases. Let us collect some of its properties:
\begin{lem}[Symmetric Dilation]\label{lem:dilation}
	Let $P \in \reals^{n_1 \times n_2}$ have a reduced SVD given by $P = U \Sigma V^T$ where $\Sigma = \diag(\sigma_1,\dots,\sigma_\kk)$ is a $\kk \times \kk$ nonnegative diagonal matrix . Then,
	\begin{enumerate}[label=(\alph*)]
		\item $\Pdag$ has a reduced EVD given by
		\begin{align*}
		\Pdag = \, W 
		\begin{pmatrix}
		\Sigma & 0 \\
		0 & -\Sigma
		\end{pmatrix} W^T, \quad 
		W = 
		\frac1{\sqrt{2}}
		\begin{pmatrix}
		U & U \\
		V & -V
		\end{pmatrix} \in \ort{(n_1+n_2)}{2\kk}.
		\end{align*}
		
		\item $P \mapsto \Pdag$ is a linear operator with
			$\mnorm{P^\dagger} = \mnorm{P}$ and $\mnorm{P^\dagger}_F = \sqrt{2}\mnorm{P}_F.$
		\item The gap between $\kk$ top (signed) eigenvalues of $P^\dagger$ and the rest of its spectrum is $2\sigma_\kk$.
	\end{enumerate}
	
\end{lem}

\begin{proof}
	Part~(a) can be verified directly (e.g. $W^TW = I_{2k}$ follows from $U^T U = V^T V = I_k$) and part~(c) follows by noting that $\sigma_j \ge 0$ for all $j$. Part~(b) also follows directly from part~(a), using unitary-invariance of the two norms.
\end{proof}

In addition, let us define a singular value (SV) truncation operator $\trunc_k : \reals^{n_1 \times n_2} \to \reals^{n_1 \times n_2}$ that takes a matrix $A$ with SVD $A = \sum_{i} \sigma_i u_i v_i^T$ to the matrix
\begin{align}\label{eq:k:reduced:SVD}
A^{(k)} := \trunc_k(A) := \sum_{i=1}^k \sigma_i u_i v_i^T.
\end{align}
In other words, $\trunc_k$ keeps the largest $k$ singular values (and the corresponding singular vectors) and zeros out the rest. Recall that we order singular values in nonincreasing fashion $\sigma_1\ge \sigma_2 \ge \cdots.$ We also refer to~\eqref{eq:k:reduced:SVD} as the \emph{$k$-truncated SVD} of $A$.
Using the dilation and the Davis--Kahan (DK) theorem for symmetric matrices~\cite[Theorem VII.3.1]{bhatia2013matrix}, we have:
\begin{lem}\label{lem:DK:Z:dev}
	Let $\Arek = \Zh_1 \Sigh \Zh_2^T$ be the $k$-truncated SVD of $\Are$ and assume that the concentration bound~\eqref{eq:gen:concent} holds. Let $\Zb_1 \Usi$ be given by the reduced SVD of $P$ in~\eqref{eq:P:svd}. Then, the deviation bound~\eqref{eq:Z:dev:bound} holds for some $k \times k$ orthogonal matrix $Q$, and $C_2 = 2C$.  
\end{lem}

Later, in Section~\ref{sec:consist:res}, we introduce an alternative spectral truncation scheme. It is worth comparing the above lemma to Lemma~\ref{lem:RR:dev} which establishes a similar result for the alternative truncation.

\begin{rem}[Symmetric case]\label{rem:sym:case}
	 When $P$ is symmetric one can still use the dilation operator. In this case, since $P$ itself is symmetric, it has an eigenvalue decomposition (EVD), say $P = U \Lambda U^T$, where $\Lambda = \diag(\lambda_1,\dots,\lambda_\kk)$ is the diagonal matrix of the eigenvalues of $P$. Since these eigenvalues could be negative, there is a slight modification needed to go from the EVD to the SVD of $P$. Let $s_i$ be the sign of $\lambda_i$ and set $S = \diag(s_i, i=1,\dots,\kk)$. Then, it is not hard to see that with $V = US$ and $\Sigma = \Lambda S = \diag(|\lambda_i|,i=1,\dots,\kk)$, we obtain the SVD $P = U \Sigma V^T$. In other words, all the discussion in this section, and in particular Lemma~\ref{lem:DK:Z:dev} hold with $V = US$ and $\sigma_i = |\lambda_i|$. The special case of Lemmas~\ref{lem:dilation} and~\ref{lem:DK:Z:dev} for the symmetric case appears in~\cite{lei2015consistency}. These observations combined with the fact that the concentration inequality discussed in Section~\ref{sec:concent} holds in the symmetric case leads to the following conclusion: All the results discussed in this paper apply to the symmetric SBM, for the version of the adjacency-based spectral clustering that \emph{sorts the eigenvalues based on their absolute values.} This is the most common version of spectral algorithms in use. On the other hand, one gets a different behavior for the algorithm that considers the top $k$ (signed) eigenvalues. We have borrowed the term ``symmetric dilation'' from~\cite{tropp2015introduction} where these ideas have been successfully used in translating matrix concentration inequalities to the symmetric case.
	 
\end{rem}


\subsection{\kmeans algorithms (Step 3)}\label{sec:kmeans:step}
Let us now give the details of the third and final step of the analysis. We introduce some notations and concepts that help in the discussion of \kmeans (type) algorithms.

\paragraph{\kmeans matrices.}
Recall that $\hr(n,\kk)$ denotes the set of \emph{hard} (cluster) labels: $\{0,1\}$-valued $n \times \kk$ matrices where each row has exactly a single 1. Take $Z \in \hr(n,\kk)$. A related notion is that of a \emph{cluster} matrix $Y = ZZ^T \in \{0,1\}^n$ where each entry denotes whether the corresponding pair are in the same cluster. Relative to $Z$, $Y$ loses the information about the ordering of the cluster labels.
We define the class of \emph{\kmeans matrices} as follows:
\begin{align}
\begin{split}\label{eq:kmm:def}
\kmm(n,\dd,\kk) &:= \{X \in \rmat(n,\dd):\; \text{$X$ has at most $\kk$ distinct rows} \} \\
&= \{Z R:\; Z \in \hr(n,\kk), R \in \rmat(\kk,\dd)\}.
\end{split}
\end{align}
The rows of $R$, which we denote as $r_i^T$, play the role of cluster centers. Let us also denote the rows of $X$ as $x_i^T$. The second equality in~\eqref{eq:kmm:def} is due to the following correspondence: Any matrix $X \in \kmm(n,\dd,\kk)$ uniquely identifies a \emph{cluster} matrix $Y \in \{0,1\}^{n \times n}$ via, $Y_{ij} = 1$ iff $x_i = x_j$. This in turn ``uniquely'' identifies a label matrix $Z$ up to $\kk!$ permutation of the labels. From $Z$, we ``uniquely'' recover $R$, with the convention of setting rows of $R$ for which there is no label equal to zero. (This could happen if $X$ has fewer than $k$ distinct rows.) 

With these conventions, there is a one-to-one correspondence between $X \in \kmm(n,\dd,\kk)$ and $(Z,R) \in \hr(n,\kk) \times \rmat(\kk,\dd)$, up to label permutations. That is, $(Z,R)$ and $(ZQ,QR)$ are considered equivalent for any permutation matrix $Q$. The correspondence allows us to talk about a (relative) misclassification rate between two \kmeans matrices: If $X_1, X_2 \in \kmmt$ with membership matrices $Z_1,Z_2 \in \hr(n,\kk)$, respectively, we set
\begin{align}
\Misb(X_1,X_2) := \Misb(Z_1,Z_2).
\end{align}

\paragraph{\kmeans as projection.} Now consider a general $\Xh \in \rmat(n,\dd)$. The classical \kmeans problem can be thought of as projecting $\Xh$ onto $\kmm(n,\dd,\kk)$, in the sense of finding a nearest member of $\kmm(n,\dd,\kk)$ to $\Xh$ in Frobenius norm. Let us write $\df(\cdot,\cdot)$ for the distance induced by the Frobenius norm, i.e., $\df(\Xh,X) = \fnorm{\Xh - X}$. The \kmeans problem is that of solving the following optimization:
\begin{align}\label{eq:kmeans:prob}
\df(\Xh,\kmm(n,\dd,\kk)) := \min_{X \,\in\, \kmm(n,\dd,\kk)} \; \df(\Xh,X).
\end{align}
The arguments to follow go through for any distance on matrices that has a $\ell_2$ decomposition over the rows:
\begin{align}\label{eq:dist:ell2:decomp}
\df(\Xh,X)^2 = \sum_{i=1}^n \dr(\xh_i,x_i)^2,
\end{align}
where $x_i^T$ and $\xh_i^T$ are the rows of $X$ and $\Xh$ respectively, and $\dr(\xh_i,x_i)$ is some distance over vectors in $\reals^d$. For the case of the Frobenius norm: $\dr(\xh_i,x_i) = \norm{\xh-x_i}_2$, the usual $\ell_2$ distance. This is the primary case we are interested in, though the result should be understood for the general case of~\eqref{eq:dist:ell2:decomp}.
Since solving the \kmeans problem~\eqref{eq:kmeans:prob} is NP-hard, one can look for approximate solutions: 
\begin{defn}
	A $\kappa$-approximate  \kmeans solution for $\Xh$ is a matrix $\Xt \in \kmm(n,\dd,\kk)$ that achieves $\kappa$ times the optimal distance:
	\begin{align}\label{eq:kappa:approx:1}
	\df(\Xh,\Xt) \le \kappa \,\df(\Xh, \kmmt).
	\end{align}
	We write $\Pc_\kappa: \rmat(n,\dd)\mapsto \kmmt$ for the (set-valued) function that maps matrices $\Xh$ to $\kappa$-approximate solutions $\Xt$.
	
\end{defn}
An equivalent restatement of~\eqref{eq:kappa:approx:1} is
\begin{align}\label{eq:kappa:approx:2}
\df(\Xh,\Xt) \le \kappa \,\df(\Xh, X), \quad \forall X \in \kmmt.
\end{align}
Note that $\Pc_\kappa(\Xh) =  \{\Xt\in \kmm(n,\dd,\kk):\; \text{$\Xt$ satisfies~\eqref{eq:kappa:approx:2}} \}$.
Our goal is to show that whenever $\Xh$ is close to some $\Xs \in \kmmt$, then any  $\kappa$-approximate \kmeans solution based on it, namely $\Xt \in \Pc_\kappa(\Xh)$ will be close to $\Xs$ as well. This is done in two steps:
\begin{enumerate}
	\item If the distance $\df(X,\Xt)$ between two \kmeans matrices $X, \Xt \in \kmmt$ is small, then their relative misclassification rate $\Misb(X,\Xt)$ is so.
	
	\item If a general matrix $\Xh \in \rmat(n,d)$ is close to a \kmeans matrix $X \in \kmmt$, then so is its $\kappa$-approximate \kmeans projection. More specifically,
	\begin{align}\label{eq:kmeans:proj:closeness}
		\df(\Xt,X) \le (1+\kappa)\, \df(\Xh,X), \quad \forall \Xt \in \Pc_\kappa(\Xh),
	\end{align}
	which follows from the triangle inequality $\df(\Xt,X) \le \df(\Xt,\Xh) + \df(\Xh,X)$ and~\eqref{eq:kappa:approx:2}.
\end{enumerate}
Combining the two steps (taking $X = \Xs$), we will have the result.

 Let us now give the details of the first step above. For this result, we need the key notion of center separation. \kmeans matrices have more information that just a membership assignment. They also contain an encoding of the relative positions of the clusters, and hence the minimal pairwise distance between them, which is key in establishing a misclassification rate.
\begin{defn}[Center separation]\label{dfn:center:sep}
	For any $X \in \kmmt$, let us denote its centers, i.e. distinct rows, as $\{q_r(X), r \in [\kk]\}$, and let	\begin{align}\label{eq:center:sep:def}
	\delta_r(X) = \min_{\ell:\; \ell \neq r} \,\dr(q_\ell(X), q_r(X)), \quad \delinf(X) = \min_r \delta_r(X).
	\end{align}
	In addition, let $n_r(X)$ be the number of nodes in cluster $r$ according to $X$, and $\ninf(X)=\min_{r} n_r(X)$, the minimum cluster size.
\end{defn}
If $X$ has $m < \kk$, the convention would be to let $q_k(X) = 0$ for $k = m+1,\dots,\kk$. We usually do not work with these degenerate cases. Implicit in the above definition is an enumeration of the clusters of $X$. We note that definition of $\delta_r = \delta_r(X)$ in~\eqref{eq:center:sep:def} implies
\begin{align}\label{eq:center:sep:alt}
	\dr(q_\ell(X) ,q_r(X)) \ge \max\{\delta_\ell,\delta_r\}, \quad \forall (r,\ell):\, r \neq \ell.
\end{align}

We are now ready to show that that any algorithm that computes a $\kappa$-approximate solution to the \kmeans problem~\eqref{eq:kmeans:prob} (where $\kappa$ is some constant) satisfies the LQC property~\eqref{eq:lqc:kmeans} needed in the last step of the analysis sketched in Section~\ref{sec:analysis:sketch}.
Recall that $\Mis_r(X;\Xt)$ is the misclassification rate over the $r$th cluster of $X$ (Section~\ref{sec:notation}). 

\begin{prop}\label{prop:kmeans:misclass}
	Let $X, \Xt \in \kmmt$ be two \kmeans matrices, and write $n_r = n_r(X)$, $\ninf = \ninf(X)$ and $\delta_r = \delta_r(X)$. Assume that $\df(X,\Xt) \le \eps$ and 
	\begin{itemize}
		\item[(a)] $X$ has exactly $\kk$ nonempty clusters, and
		\item[(b)] $c_r^{-2}\,\eps^2 /(\delta_r^2 n_r) < 1$ for $r \in [\kk]$, and constants $c_r > 0$ such that $c_r + c_\ell \le 1, \; r \neq \ell$.
	\end{itemize}
	Then, $\Xt$ has exactly $\kk$ clusters and
	\begin{align}
		\Mis_r(X;\Xt) \le \frac{c_r^{-2} \,\eps^2}{n_r \,\delta_r^2}, \quad \forall r \in [\kk].
	\end{align}
\end{prop}
In particular, under the conditions of Proposition~\ref{prop:kmeans:misclass} with $c_r = 1/2$, we have
\begin{align*}
\Misinf(X,\Xt) \le \frac{4 \,\eps^2}{\min_r n_r \delta_r^2} \le \frac{4 \,\eps^2}{\ninf \delinf^2}, \qquad \Misb(X,\Xt) \le \frac{4\, \eps^2}{ n \,\delinf^2}.
\end{align*}
where the second one follows from the identity $ \Misb(X,\Xt) = \sum_{r=1}^{\kk}(n_r/n) \Mis_r(X,\Xt)$.
Combining Proposition~\ref{prop:kmeans:misclass} with~\eqref{eq:kmeans:proj:closeness}, we obtain the following corollary:
\begin{cor}\label{cor:kmeans:misclass}
		Let $\Xs \in \kmmt$ be a \kmeans matrix, and write $n_r = n_r(\Xs)$, $\ninf = \ninf(\Xs)$ and $\delta_r = \delta_r(\Xs)$. Assume that $\Xh \in \rmat(n,\dd)$ is such that $\df(\Xs, \Xh) \le \eps$ and 
	\begin{itemize}
		\item[(a)] $\Xs$ has exactly $\kk$ nonempty clusters, and
		\item[(b)] $c_r^{-2}\,(1+\kappa)^2\eps^2 /(\delta_r^2 n_r) < 1$ for $r \in [\kk]$, and constants $c_r > 0$ such that $c_r + c_\ell \le 1, \; r \neq \ell$.
	\end{itemize}
	Then, any $\Xt \in \Pc_\kappa(\Xh)$ has exactly $\kk$ clusters and
	\begin{align}
	\Mis_r(\Xs;\Pc_\kappa(\Xh)) \le \frac{c_r^{-2} \,(1+\kappa)^2\eps^2}{n_r \,\delta_r^2}, \quad \forall r \in [\kk].
	\end{align}
\end{cor}
As before, $\Mis_r(\Xs,\Pc_\kappa(\Xh))$ should be interpreted as $\max_{\Xt \in \Pc_\kappa(\Xh)} \Mis_r(\Xs,\Xt)$, that is, the result hold for any $\kappa$-approximate \kmeans solution for $\Xh$. 
In particular, under the conditions of Corollary~\ref{cor:kmeans:misclass} with $c_r = 1/2$, we have
\begin{align}\label{eq:Misinf:Misb:approx:kmeans}
	\Misb(\Xs,\Pc_\kappa(\Xh)) &\;\le\; 4 (1+\kappa)^2 \frac{\eps^2}{ n \,\delinf^2}.
\end{align}
A similar bound holds for $\Misinf$. Note that~\eqref{eq:Misinf:Misb:approx:kmeans} is of the desired form needed in~\eqref{eq:lqc:kmeans}.

\begin{rem}\label{rem:kmeans}
	Corollary~\ref{cor:kmeans:misclass} shows that any constant-factor approximation to the \kmeans problem~\eqref{eq:kmeans:prob} satisfies the LCQ property~\eqref{eq:lqc:kmeans}. Such approximations can be computed in polynomial time; see for example~\cite{kumar2004simple}. One can also use Lloyd's algorithm with \verb|kmeans++| initialization to get a $\kappa \lesssim \log k$ approximation~\cite{arthur2007k}. Both~\cite{kumar2004simple,arthur2007k} give constant probability approximations, hence if such algorithms are used in our subsequent results, ``w.h.p.'' should be interpreted as with high constant probability (rather than $1-o(1)$). Recently~\cite{lu2016statistical} has shown that Lloyd's algorithm with random initialization can achieve near optimal misclassification rate in certain random models. Their analysis is complementary to ours in that we establish and use the LQC property~\eqref{eq:lqc:kmeans} which uniformly holds for any input matrix $\Xh$ and we do not consider specific algorithms. The core ideas of our analysis in this section are borrowed from~\cite{lei2015consistency,jin2015fast}.
	It is worth noting that there are other algorithms that turn a general matrix into a \kmeans matrix, without trying to approximate the \kmeans problem~\eqref{eq:kmeans:prob}, and still satisfy the LQC. We will give one such example, Algorithm~\ref{alg:kmeans:replace}, in Appendix~\ref{app:kmeans:replace}. Interestingly, in contrast to the \kmeans approximation algorithms, the LQC guarantee for Algorithm~\ref{alg:kmeans:replace} is not probabilistic.
\end{rem}

\section{Regularization and concentration}\label{sec:concent}

Here, we provide the details of Step~1, namely, the concentration of the regularized adjacency matrix. 
We start by a slight generalization of the results of~\cite{le2017concentration} to the rectangular case:

\begin{thm}\label{thm:concent:nonsym}
	Assume $n_1 \le n_2$ and let $A \in \{0,1\}^{n_1 \times n_2}$ have independent Bernoulli entries with mean $\ex[A_{ij}] = p_{ij}$. Take $d \ge \max_{ij} n_2\, p_{ij}$. Pick any subsets $\Ic_1 \subset [n_1]$ and $\Ic_2 \subset [n_2]$ such that
	\begin{align}\label{eq:I1:I2:defs}
		 |\Ic_1| \le 10 n_2 / d, \quad \text{and} \quad |\Ic_2| \le 10n_2/d
	\end{align}
	and fix some $d' > 0$. Define a regularized adjacency matrix $\Are$ as follows: 
	\begin{enumerate}
		\item[(a)] Set $[\Are]_{ij} = A_{ij}$ for all $(i,j) \in \Ic_1^c \times \Ic_2^c$.
		\item[(b)] Set $[\Are]_{ij}$ arbitrarily when $i \in \Ic_1$ or $j \in \Ic_2$, but subject to the constraints $\norm{[\Are]_{i*}}_1 \le d'$ and $\norm{[\Are]_{*j}}_1 \le d'$ for all $i \in \Ic_1$ and $j \in \Ic_2$.
	\end{enumerate}
	Then, for any $r \ge 1$,
	with probability at least $1-n_2^{-r}$, 
	the new adjacency matrix $\Are$ satisfies
	\begin{align}
	\opnorm{\Are - \ex [A] } \le C_2\, r^{3/2} (\sqrt{d} + \sqrt{d'}).
	\end{align}
	The same result holds if in step~(b) one uses $\ell_2$ norm instead of $\ell_1$.
\end{thm}

Theorem~\ref{thm:concent:nonsym} follows directly from the non-symmetric (i.e., directed) version of~\cite[Theorem~2.1]{le2017concentration}, by padding $A$ with rows of zeros to get a square $n_2 \times n_2$ matrix. The result then follows  by the same argument as in \cite[Theorem~2.1]{le2017concentration}. The term ``arbitrary'' in the statement of the theorem includes any reduction even if the scheme is stochastic and depends on $A$ itself. This feature will be key in developing data-driven schemes.  


To apply Theorem~\ref{thm:concent:nonsym}, take $\Ic_1 = \{i \in [n_1]: \sum_{j} A_{ij} > 2d\}$ and $\Ic_2 = \{j \in [n_2] : \sum_{i} A_{ij} > 2d\}$, i.e., the set of rows and columns with degrees larger than $2d$.
It is not hard to see that if $d$ is any upper bound on the expected row and column degrees, then the sizes of these sets satisfy~\eqref{eq:I1:I2:defs} with high probability. This follows, for example, from the same argument as that leading to~\eqref{eq:gamma:dev} in Appendix~\ref{sec:proof:deg:reg}. Recalling the scaling of the connectivity matrix in~\eqref{eq:nonsym:mean:def}, taking $d = n_2 \infnorm{P}$ gives an upper bound on the expected row and column degrees, assuming $n_2 \ge n_1$. Thus, we can apply 
Theorem~\ref{thm:concent:nonsym} with $ d' = d = \sqrt{n_2 / n_1} \infnorm{\Psi}  = n_2 \infnorm{P}  $, 
to obtain the desired concentration bound~\eqref{eq:gen:concent}. Note that this concentration result does not require $d = n_2 \infnorm{P}$ to satisfy any lower bound (such as $d = \Omega (\log n_2)$). See Remarks~\ref{rem:comparison:reg} and~\ref{rem:comparison:consistent} for the significance of this fact.

A disadvantage of the regularization scheme described in Theorem~\ref{thm:concent:nonsym} is the required  knowledge of a good upper bound on $d =n_2 \infnorm{P} $. In the next section, under a SBM with some mild regularity assumptions, we develop a fully data-driven  scheme with the same guarantees as those of Theorem~\ref{thm:concent:nonsym}.

\subsection{Data-driven regularization}\label{sec:data:driven:trunc}

We now describe a regularization scheme that is data-driven and does not require the knowledge of $d$ as in Theorem~\ref{thm:concent:nonsym}. Let us write $D_i = \sum_{j=1}^{n_2} A_{ij}$ for the degree of node $i$ on side~1 and $\Db = \frac1{n_1} \sum_{i=1}^{n_1}D_i$ for the average degree on side~1. Consider the following order statistics for the degree sequence:
\begin{align}
D_{(1)} \ge D_{(2)}\ge \dots \ge D_{(n_1)}.
\end{align}
The idea is that under a block model, we can achieve concentration~\eqref{eq:gen:concent} by reducing the row degrees that are roughly above $D_{(\alpha)}$ for $\alpha = \lfloor n_1/\Db\rfloor$ (and similarly for the columns). The overall  scheme is described in Algorithm~\ref{alg:deg:reg}. The algorithm has a tuning parameter $\tau$; however, we will show that taking   $\tau = 3$ is enough. This constant is for convenience and has no special meaning. 
Note that the regularization described in Step~\ref{step:set:Are2} is a special case of that in Step~\ref{step:set:Are}; it is equivalent to first reducing the row degrees by setting $[\Are]_{i*} \gets A_{i*} \dhat_1 / D_i$ for $i \in \Ich_1$ followed by reducing the column degrees $[\Are]_{*j} \gets [\Are]_{*j} \dhat_2 / D_j$ for $j \in \Ich_2$, or vice versa. Alternatively, we can replace the $\ell_1$ norm constraints in Step~\ref{step:set:Are} with $\ell_2$ norm version and replace Step~\ref{step:set:Are2} with  $[\Are]_{ij} =  A_{ij} \sqrt{w_i w'_j}$ for all $i\in[n_1], j\in [n_2]$.

\begin{algorithm}[t]
	\setstretch{1.4}
	\hspace*{\algorithmicindent} \textbf{Input:}  Adjacency matrix $A \in \reals^{n_1 \times n_2}$ and regularization parameter $\tau$. (Default: $\tau = 3$) \\
	\hspace*{\algorithmicindent} \textbf{Output:} Regularized adjacency matrix $\Are \in  \reals^{n_1 \times n_2}$.
	\begin{algorithmic}[1]
		\State  Form the degree sequence $D_i = \sum_{j=1}^{n_2} A_{ij}$ for $i=1,\dots,n$ and the corresponding order statistics: $D_{(1)} \ge D_{(2)}\ge \dots \ge D_{(n_1)}$.\label{step:1}
		\State Let $\Db = \frac1{n_1} \sum_{i=1}^{n_1}D_i$ and $\alpha = \lfloor n_1/\Db\rfloor$.
		\State Set $\dhat_1 = \tau D_{(\alpha)}$ and $\Ich_1 = \{i: D_i \ge \dhat_1 \}$. \label{step:3}
		\State Repeat Steps~\ref{step:1}--\ref{step:3} on $A^T$ to get $\Ich_2$ and $\dhat_2$, and let $D'_j$ be the corresponding degrees (i.e., column degrees in the original matrix $A$).
		\State 	Set $[\Are]_{ij} = A_{ij}$ for all $(i,j) \in \Ich_1^c \times \Ich_2^c$,
		and set $[\Are]_{ij}$ arbitrarily when $i \in \Ich_1$ or $j \in \Ich_2$, but subject to: 
				$\norm{[\Are]_{i*}}_1 \le \dhat_1$ for all $i \in \Ich_1$ and $\norm{[\Are]_{*j}}_1 \le \dhat_2$ for all $j \in \Ich_2$. \label{step:set:Are}
		\State For example, $[\Are]_{ij} = A_{ij} w_i w'_j$ for all $(i,j) \in [n_1] \times [n_2]$, where $ w_i = \frac{\dhat_1}{D_i} \wedge 1$ and $w'_j= \frac{\dhat_2}{D'_j} \wedge 1$ satisfies the conditions in Step~\ref{step:set:Are}.\label{step:set:Are2}
	\end{algorithmic}
	\caption{Data-driven adjacency regularization}
	\label{alg:deg:reg}
\end{algorithm}

\smallskip
Let $p_{ij} = \ex[A_{ij}]$ and define $d_i := \ex[ D_i] = \sum_j p_{ij}$ and
\begin{align}\label{eq:d:defs}
\db = \frac1{n_1} \sum_{i=1}^n d_i, \quad \dmax := \max_i d_i, \quad d := n_2 \cdot \max_{i,j} p_{ij}.
\end{align}
Note that $\db \le \dmax \le d$ and we have $\db = \ex[\Db]$, that is, the expected average degree of the network (on side~1). We need the following mild regularity assumption:
\begin{enumerate}[start=2, label=(A\arabic*)]
	\item \label{assum:deg} Consider the SBM model given by~\eqref{eq:nonsym:mean:def} and~\eqref{eq:matrix:ber:defn} with $n_1 \le n_2$.  
	%
	Assume further that for some $\beta \ge  1$, and for all $t \in [k_1]$ and $s \in [k_2]$,
	\begin{align}\label{eq:deg:trunc:conds}
	n_{1t} \ge \frac{n_1 }{ \beta k_1},\quad n_{2s} \ge \frac{n_2} {\beta k_2}, 
	\quad \text{and} \quad \Big(\frac{n_2}{n_1}\Big)^2 \max\{8\beta k_1, 8\beta k_2, 90\}  \;\le \;\db\;  \le\; \frac12 n_1.
	\end{align}
\end{enumerate}
The next key lemma shows that $D_{(\alpha)}$ is a good proxy for $\dmax$ and $d$ under a SBM:
\begin{lem}\label{lem:deg:trunc}

	Let $\alpha = \lfloor n_1/\Db\rfloor$ where $\Db$ is the average degree on side~1 as defined earlier. Under assumption~\ref{assum:deg}, the following hold, with probability at least $1-3 e^{- n_1/80}-e^{-n_1/(\beta k_2)}$, 
	\begin{itemize}
		\item [(a)] $d_{\max}/2\le D_{(\alpha)}\le 3d_{\max}/2$.
		\item [(b)] $\big|i: D_i>3D_{(\alpha)}\big| \le 10 n_1/ d$.
	\end{itemize}	
	
\end{lem}
Conditions in~\eqref{eq:deg:trunc:conds} are quite mild: The first and second require that the clusters associated with the maximum  degree are not too small. In the balanced case, where all clusters are of equal size, we have $n_{rt} = n_r/k_r$ (for all $t \in [k_r]$ and $r=1,2$) hence the condition is satisfied with $\beta=1$. In general, $\beta$ measures deviations from balancedness. Note that \eqref{asu:clust:prop} and the first two conditions in \eqref{eq:deg:trunc:conds} do not imply each other. The last condition in~\eqref{eq:deg:trunc:conds} is satisfied if $\db \to \infty$ slower than $n_1$ and faster than $\max\{k_1, k_2\}$ (assuming $\beta = O(1)$ and $n_1 \asymp n_2$).


A result similar to Lemma~\ref{lem:deg:trunc} holds for the column degrees with appropriate modifications: Let $D'_j = \sum_{i} A_{ij}$ be the column degrees with expectation $d'_j = \sum_i p_{ij}$, and define the corresponding average $\db' = \frac1{n_2} \sum_j d'_j$ and $\dmax' = \max_j d'_j$. We, however, keep $d$ as defined in~\eqref{eq:d:defs}.
\begin{lem}\label{lem:deg:trunc:col}
	Let $\alpha' = \lfloor n_2/\Db'\rfloor$ where $\Db'$ is the average degree on side~2 as defined earlier. Under assumption~\ref{assum:deg}, 
	the following hold, with probability at least $1-3 e^{- n_2/80}-e^{-n_1/(\beta k_1)}$, 
	\begin{itemize}
		\item [(a)] $d_{\max}'/2\le D'_{(\alpha')}\le 3d'_{\max}/2$.
		\item [(b)] $\big|j: D'_j>3D'_{(\alpha')}\big| \le 10 n_2/ d$.
	\end{itemize}	
	
\end{lem}

Equipped with Lemmas~\ref{lem:deg:trunc} and \ref{lem:deg:trunc:col}, we have the following  for Algorithm~\ref{alg:deg:reg}:
\begin{thm}\label{thm:concent:data:driven}
	Assume that $A \in \reals^{n_1 \times n_2}$ is generated from the SBM model satisfying~\ref{assum:deg}. 
	Then, for any $r \ge 1$, with probability at least $1-n_2^{-r} - 6 e^{-n_1/80} - 2 e^{-n_1 / \beta \max\{k_1, k_2\}}$,
	the regularized output $\Are$ of Algorithm~\ref{alg:deg:reg} (with $\tau = 3$) satisfies
	\begin{align}\label{eq:op:diff:bound}
	\opnorm{\Are - \ex [A] } \le C_3 r^{3/2}\sqrt{\dg},
	\end{align} 
	where $\dg$ is as in~\eqref{eq:d:defs} and $C_3 > 0$ is a universal constant.  
\end{thm}

\begin{proof}[Proof of Theorem~\ref{thm:concent:data:driven}]
	Recall the definitions of $\Ich_r$ and $\dhat_r$ ($r=1,2$) in Algorithm~\ref{alg:deg:reg}. On the union of the events described in Lemmas~\ref{lem:deg:trunc} and~\ref{lem:deg:trunc:col}, both $|\Ich_1|$ and $|\Ich_2|$ are bounded by $10n_2/d$ (since $n_1 \le n_2$), $\dhat_1 = 3 D_{(\alpha)} \le 9\dmax/2$ and $\dhat_2 = 3 D'_{(\alpha')} \le 9\dmax'/2$.
	Since  $\dmax \le n_2 \infnorm{P} = d$ and $\dmax' \le n_1 \infnorm{P} \le d$, we can apply Theorem~\ref{thm:concent:nonsym} with $d' = 9d/2$ completing the proof.
\end{proof}
Theorem~\ref{thm:concent:data:driven} thus provides the same concentration guarantee as in Theorem~\ref{thm:concent:nonsym} without the knowledge of $d$. 
See Appendix~\ref{sec:proof:deg:reg} for the proof of the two lemmas of this section.

\begin{rem}[Comparison with existing work]\label{rem:comparison:reg}
	Results of the form~\eqref{eq:gen:concent} 
	hold for $A$ itself without any regularization if one further assumes that $d \gtrsim \log n_2$. These result are often derived for the symmetric case; see for example~\cite{Tomozei2011,lei2015consistency,chen2016statistical} or~\cite{bandeira2016sharp} for the more general result with $d = \max_i \sum_{j} p_{ij}$. In order to break the $\log n_2$ barrier on the degrees, one has to resort to some form of regularization.  
	When $d$ is given, the general regularization for the adjacency matrix is to either to remove the high degree nodes as in~\cite{chin2015stochastic} or reduce their effect as in~\cite{le2017concentration} and Theorem~\ref{thm:concent:nonsym} above. 
	In contrast, there is little work on data-driven regularization for the adjacency matrix.
	One such algorithm was investigated in~\cite{gao2017achieving}
	for the special case of the SPBB model---discussed in Example~\ref{exa:sbpp} (Section~\ref{sec:consist:res})---under the assumption $a=O(b)$. The algorithm truncates the degrees to a multiple of the average degree, i.e., $\dhat_1 = C \Db$ for some large $C > 0$.
	A possible choice of $C=5$ is given in \cite{yun2014accurate} assuming that the expected degrees of the nodes are similar. 
	
	In contrast to existing results, our data-driven regularization holds for a general SBM and only requires the mild assumptions in~\ref{assum:deg} while preserving the same upper bound~\eqref{eq:op:diff:bound} that holds with the knowledge of $d$. In particular, we do not require $\min_{k\ell} \Psi_{k\ell} \asymp \max_{k\ell} \Psi_{k\ell}$ (which is what $a = O(b)$ means in the SPBB model). Algorithm~\ref{alg:deg:reg} enables effective and provable regularization without knowing any parameters.
	
\end{rem}

\section{Consistency results}\label{sec:consist:res}
We now state our various consistency results. The proofs are deferred to Appendix~\ref{sec:proof:consist:res}. We start with a refinement of Theorem~\ref{thm:sc:prototype} for the specific algorithm \scone given in Algorithm~\ref{alg:scone}.

\begin{thm}\label{thm:scone}
	Consider the spectral algorithm \scone given in Algorithm~\ref{alg:scone}. Assume $\kk_1 = \kk =: \min\{\kk_1,\kk_2\}$, and for a sufficiently small $C > 0$,
	\begin{align*}
		\kk \dg \, \sigma_\kk^{-2} \le C(1+\kappa)^{-2}.
	\end{align*}
	  Then, under the SBM model satisfying~\ref{assum:deg}, w.h.p.,
	\begin{align*}
	\Misb(\Pc_\kappa(\Zh_1),\Zb_1) \; \lesssim \; (1+\kappa)^2 \beta_1 \Big(\frac{\dg}{\sigma_\kk^2}\Big).
	\end{align*}
	where $\beta_1$ is given in~\eqref{asu:clust:prop} and $\dg$ is defined in~\eqref{eq:gen:concent}.
\end{thm}

One can take $\kappa$ to be a fixed small constant say $1.5$, since there are $\kappa$-approximate \kmeans algorithms for any $\kappa > 1$. In that case, $(1+\kappa)^2$ can be absorbed into other constants, and the bound in Theorem~\ref{thm:scone} is qualitatively similar to Theorem~\ref{thm:sc:prototype}.

\begin{algorithm}[t]
	\setstretch{1.2}
	\caption{\scone}
	\begin{algorithmic}[1]
		\medskip
		\State  Apply  regularization Algorithm~\ref{alg:deg:reg} to $A$ to obtain $\Are$.
		\State Obtain the $\kk$-truncated SVD of $\Are$ as $\Arek = \Zh_1 \Sigh \Zh_2^T$. See~\eqref{eq:k:reduced:SVD}.
		\State Output an element of $\Pc_\kappa(\Zh_1)$, i.e., a $\kappa$-approximate \kmeans solution for input $\Zh_1$.
	\end{algorithmic}
	\label{alg:scone}
\end{algorithm}

\begin{algorithm}[t]\label{alg:scrr}
	\setstretch{1.2}
	\caption{\scrr}
	\begin{algorithmic}[1]
		\medskip
		\State Apply  regularization Algorithm~\ref{alg:deg:reg} to $A$ to obtain $\Are$.
		\State Obtain the best rank $\kk$ approximation of $\Are$, that is, $\Arek = \trunc_\kk(\Are)$. See~\eqref{eq:k:reduced:SVD}.
		\State Output an element of $\Pc_\kappa(\Arek)$, i.e., a $\kappa$-approximate \kmeans solution for input $\Arek$.
	\end{algorithmic}
	\label{alg:scrr}
\end{algorithm}

\subsection{Reduced-rank SC} 
Theorem~\ref{thm:scone} implicitly assumes $\sigma_k > 0$, otherwise the bound is vacuous. This assumption is clearly violated if $B$ is rank deficient (or equivalently, $P$ has rank less than $k$). A variant of SC suggested in~\cite{yun2014accurate,gao2018community} can resolve this issue. The idea is to use the entire rank $\kk$ approximation of $\Are$, and not just the singular vector matrix $\Zh_1$, as the input to the \kmeans step. This approach, which we call reduced-rank SC, or \scrr, is detailed in Algorithm~\ref{alg:scrr}. Recall the SV truncation operator $\trunc_\kk$ given in~\eqref{eq:k:reduced:SVD}. It is well-known that $\trunc_\kk$ maps every matrix to its best rank-$\kk$ approximation in Frobenius norm, i.e.,
\begin{align*}
\trunc_\kk(\Are) = \min \big\{\fnorm{R - \Are}:\; \rank(R) \le \kk \big\}
\end{align*}
with the approximation error satisfying
\begin{align}\label{eq:approx:err:k:reduced:SVD}
\opnorm{\trunc_\kk(\Are) - \Are} = \sigma_{\kk+1}(\Are).
\end{align}
\scrr uses this best rank-$\kk$ approximation as a denoised version of $\Are$ and runs a \kmeans algorithm on its rows. 
 To analyze \scrr, we need to replace bound~\eqref{eq:Z:dev:bound} in Step~2 with an appropriate modification. The following lemma replaces Lemma~\ref{lem:DK:Z:dev} and provides the necessary bound in this case. 

\begin{lem}\label{lem:RR:dev}
	Let $\Arek = \trunc_\kk(\Are)$ be the $k$-truncated SVD of $\Are$ and assume that the concentration bound~\eqref{eq:gen:concent} holds. Then,
	\begin{align}\label{eq:RR:dev}
		\fnorm{\Arek - P} \le C \sqrt{8\,\kk \dg}.
	\end{align}
\end{lem}

	Comparing with~\eqref{eq:Z:dev:bound}, we observe that~\eqref{eq:RR:dev} provides an improvement by removing the dependence on the singular value gap $\sigma_{\kk}$. However, we note that in terms of the relative error, i.e., $\fnorm{\Arek - P}/\fnorm{P}$ this may or may not be an improvement. There are cases where $\fnorm{P} \approx \sqrt{\kk} \sigma_\kk$, in which case the relative error predicted by~\eqref{eq:RR:dev} is $O(\sqrt{\dg}/\sigma_k)$,  similar to the relative error bound based on~\eqref{eq:Z:dev:bound} (since $\norm{\Zb_1 \Usi Q}_F = \sqrt{k}$); see Example~\ref{exa:sbpp} below.


	\medskip
	Following through the three-step analysis of Section~\ref{sec:analysis:sketch}, with~\eqref{eq:Z:dev:bound} replaced with~\eqref{eq:RR:dev}, we obtain a  qualitatively different bound on the misclassification error of Algorithm~\ref{alg:scrr}. The key is that center separation of $P$ treated as a \kmeans matrix is different from that of $\Zb_1 \Usi Q$. Note that $P$ is indeed a valid \kmeans matrix according to Definition~\eqref{eq:kmm:def}; in fact, $P \in \kmm(n_1,n_2,\kk_1)$. Similarly, $P^T \in \kmm(n_2,n_1,\kk_2)$. Let us define
	\begin{align}\label{eq:Psinf:defs}
	\Psinf[1]^2 &:= \min_{(s,t):\, s \neq t} \sum_{\ell=1}^{\kk_2} \pi_{2 \ell} ( \Psi_{ s \ell} - \Psi_{t \ell})^2,  \\
	\Psinft[1]^2 &:= \min_{(s,t):\, s \neq t} \Big[ \pi_{1t} \sum_{\ell=1}^{\kk_2} \pi_{2 \ell} ( \Psi_{ s \ell} - \Psi_{t \ell})^2 \Big].\label{eq:Psinft:defs}
	\end{align}
	
	\begin{thm}\label{thm:scrr}
		Consider the spectral algorithm \scrr given in Algorithm~\ref{alg:scrr}. Assume that for a sufficiently small $C_1 > 0$,
		\begin{align}\label{eq:scrr:asump}
			k \dg \,\Psinft[1]^{-2} \le C_1 (1+\kappa)^{-2}.
		\end{align}
		Then, under the SBM model satisfying~\ref{assum:deg}, with $\dg$ as defined in~\eqref{eq:gen:concent}, w.h.p.,
		\begin{align*}
		\Misb(\Pc_\kappa(\Arek),P) \; \le  \;C_1^{-1}\, (1+\kappa)^2 \Big(\frac{ k \dg}{\Psinf[1]^2}\Big).
		\end{align*}
	\end{thm}

	As is clear from the proof, one can take $C_1 = 1/(32C^2)$ where $C$ is the constant in concentration bound~\eqref{eq:gen:concent}. Condition~\eqref{eq:scrr:asump} can be replaced with the stronger assumption
	\begin{align}
		k \dg \, (\pinf{1} \Psinf[1]^2)^{-1} \le C_1 (1+\kappa)^{-2}
	\end{align}
a	where $\pinf{1} := \min_{t \,\in\, [\kk_1]} \pi_{1t}$, since $\Psinft[1]^2 \ge \pinf{1} \Psinf[1]^2$.
	
	Although the bounds of Theorems~\ref{thm:scone} and~\ref{thm:scrr} are different, surprisingly, in the case of the planted partition model, they give the same result as the next example shows.
	\begin{exa}[Planted partition model, symmetric case]\label{exa:sbpp}
	Let us consider the simplest symmetric SBM, the symmetric balanced planted partition (SBPP) model, and consider the consequences of Theorems~\ref{thm:scone} and~\ref{thm:scrr} in this case. Recall that in the symmetric case we drop index $r$ from $\kk_r$, $n_r$, $n_{rj}$, $\Nb_r$, $\beta_r$, $\Psinf[r]$ and so on. SBPP is characterized by the following assumptions:
	\begin{align*}
		\Psi = b \onem_\kk + (a-b) I_\kk,\;\;  a \ge b,\quad \pi_{j} = n_{j}/n = \frac1\kk, \; \forall j \in [\kk].
	\end{align*}	
	Here, $\onem_\kk \in \rmat(\kk,\kk)$ is the all ones matrix and \emph{balanced} refers to all the communities being of equal size, leading to cluster proportions $\pi_j = 1/\kk$. In particular, $\beta = 1$, as defined in~\eqref{asu:clust:prop}. We have 	$\Bb = \Nb^{1/2} \Psi \Nb^{1/2} = \Psi / \kk$,
	recalling $\Nb = \diag(\pi_j)$. Hence, the smallest singular value of $\Bb$ is $\sigma_k = (a-b)/\kk$. Theorem~\ref{thm:scone} gives the following result:
	\begin{cor}\label{cor:sbpp:scone}
		Under the SBPP model, as long as $k^3 a /(a-b)^2$ is sufficiently small, \scone has average misclassification error of $O( \kk^2 a/(a-b)^2)$ with high probability.
	\end{cor}
	
	Now consider \scrr. Using definitions~\eqref{eq:Psinf:defs}, we have $\kk \Psinft^2 = \Psinf^2 =  2(a-b)^2/\kk$. Then, Theorem~\ref{thm:scrr} gives  the exact same result for \scrr:	
	\begin{cor}\label{cor:sbpp:scrr}
		Corollary~\ref{cor:sbpp:scone} holds with \scone replaced with \scrr.
	\end{cor}
	Results of Corollary~\ref{cor:sbpp:scone} and~\ref{cor:sbpp:scrr} are consistency results as long as $\kk^2 a /(a-b)^2 = o(1)$. A typical example is when $\kk = O(1)$, $a = a_0 f_n$, $b = b_0 f_n$, $a_0 \asymp 1$ and $b_0 \asymp 1$ for some $f_n \to \infty$ as $n \to \infty$. Then, \scone and \scrr are both consistent at a rate $O(f_n^{-1})$.
	\qed 
	
	\end{exa}
	\begin{figure}[t]
		\centering
		\includegraphics[scale=0.5]{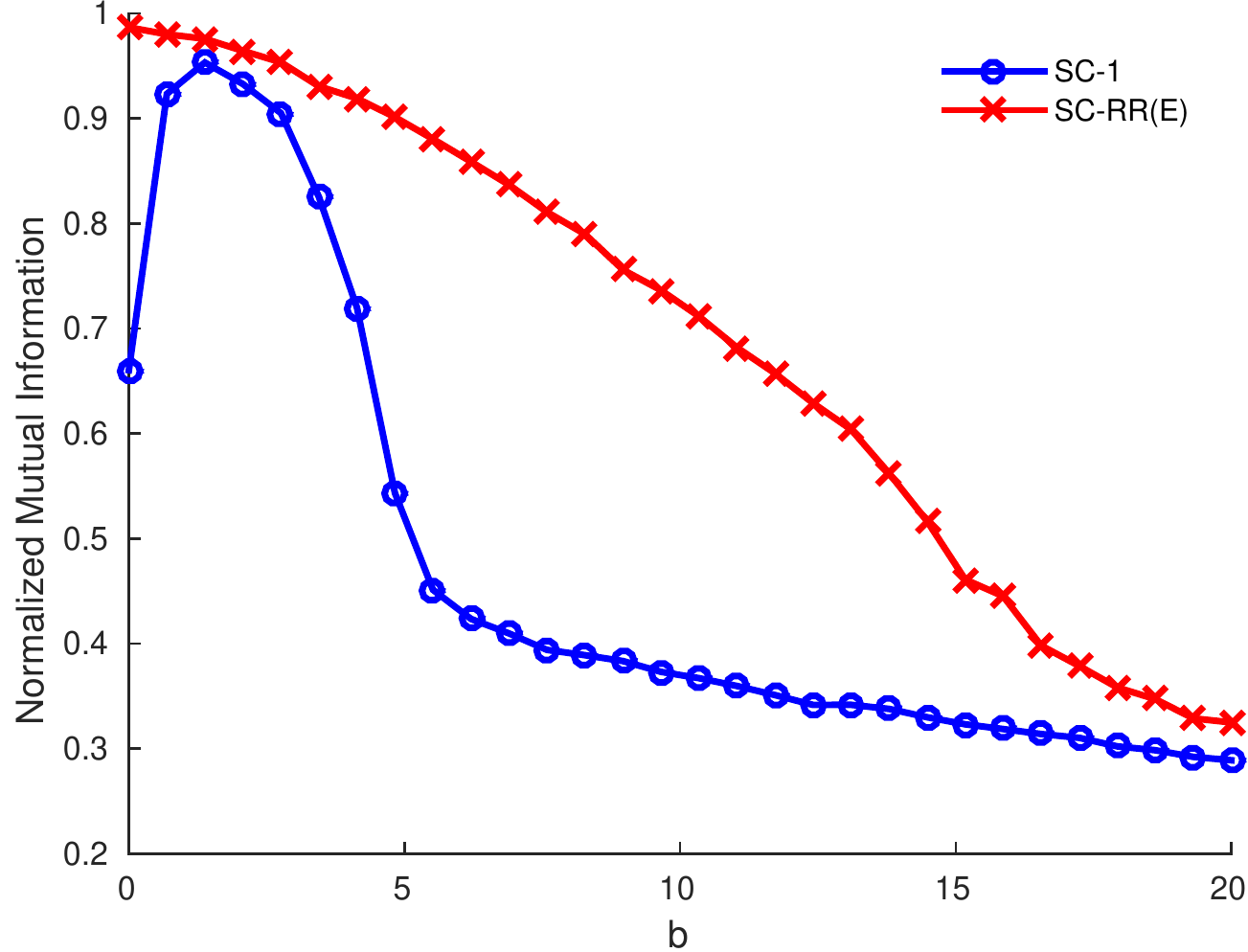}
		\caption{An example of the performance boost of \scrr (or \scerr) relative to \scone. The data is generated from the bipartite version of Example~\ref{exa:diff} with $n_2 = 2 n_1 = 1000$, $\kk_1 = \kk_2 = 4$, $\pi_{r\ell} = n_{r}/\kk_r$ for all $\ell \in [k_r],\,r=1,2$, and $\Psi = 2b E_4 + \diag(16,16,16,2)$ similar to~\eqref{eq:Psi:unequal:diag}. The key is the significant difference in the two smallest diagonal elements of $\Psi$. The plot shows the normalized mutual information (a measure of cluster quality) between the output of the two spectral clustering algorithms and the true clusters, as $b$ varies. Only row clusters are considered. The plot shows a significant improvement for \scrre relative to \scone over a range of $b$. As $b$ increases, the relative difference between $\Psi_{33}$ and $\Psi_{44}$ reduces and the model approaches that of Example~\ref{exa:sbpp}, leading to similar performances for both algorithms as expected. It is interesting to note that the monotone nature of the performance of \scrre as a function of $b$ and the non-monotone nature of that of \scone is reflected in the upper bounds~\eqref{eq:rho:scrr} and~\eqref{eq:rho:scone}.}
		\label{fig:scrr:scone:boost}
	\end{figure}
	Let us now give an example where \scone and \scrr behave differently.
	\begin{exa}\label{exa:diff}
		Consider the symmetric balanced SBM, with
		\begin{align}\label{eq:Psi:unequal:diag}
			\Psi = b \onem_\kk + \diag(\alpha_1,\dots,\alpha_k), \quad \pi_{j} = n_{j}/n = \frac1\kk, \; \forall j \in [\kk].
		\end{align}
		As in Example~\ref{exa:sbpp}, we have dropped the index $r$ determining the side of network in the bipartite case. Let us assume that $\alpha_1 \ge \alpha_2 \ge \dots \ge \alpha_\kk \ge 0$. We have
		\begin{align*}
			\kk \Psinft^2 = \Psinf^2 = \kk^{-1} \min_{s \neq t} \sum_{\ell} (\Psi_{s\ell} - \Psi_{t\ell})^2 = \kk^{-1} (\alpha_{\kk-1}^2 + \alpha_{\kk}^2).
		\end{align*}
		Thus, Theorem~\ref{thm:scrr} gives the following:
			With $\rho$ defined as follows:
			\begin{align}\label{eq:rho:scrr}
				\rho := \kk^2 \frac{\alpha_1 + b}{\alpha_{\kk-1}^2 + \alpha_{\kk}^2},
			\end{align}
			as long as $\kk \rho$ is sufficiently small, \scrr has average misclassification error $O(\rho)$ with high probability.
		
		To determine the performance of \scone, we need to estimate $\sigma_\kk$, the smallest singular value of  $\Bb = \Nb^{1/2} \Psi \Nb^{1/2} = \Psi / \kk$. Since $\Psi$ is obtained by a  rank-one perturbation of a diagonal matrix, it is well-known that  when $\{\alpha_t\}$ are distinct, the eigenvalues of $\Psi$ are obtained by solving $\sum_{t=1}^\kk 1/(\alpha_t - \lambda)= -1/b$; the case where some of the $\{\alpha_t\}$ are repeated can be reasoned by the taking the limit of the general case. By plotting $\lambda \mapsto \sum_{t=1}^\kk 1/(\alpha_t - \lambda)$ and looking at the intersection with $\lambda \mapsto -1/b$, one can see that the smallest eigenvalue of $\Psi$, equivalently its smallest singular value,  is in $[\alpha_\kk,\alpha_{\kk-1}]$, and can be made arbitrarily close to $\alpha_k$ by letting $b \to 0$. Letting $\alpha_\kk + \eps_\kk(\alpha;b)$ denote this smallest singular value, we have $0 \le \eps_k(\alpha;b) \to 0$ as $b \to 0$. 
		
		It follows that $\sigma_\kk = \sigma_\kk(\Bb) = \kk^{-1}(\alpha_\kk + \eps_k(\alpha;b))$. Theorem~\ref{thm:scone} gives
		the following: 
			With $\rho$ defined as 
			\begin{align}\label{eq:rho:scone}
			\rho := \kk^2 \frac{\alpha_1 + b}{(\alpha_{\kk} + \eps_k(\alpha;b) )^2},
			\end{align}
			as long as $\kk \rho$ is sufficiently small, \scone has average misclassification error $O(\rho)$ with high probability.
		
		Comparing~\eqref{eq:rho:scone} with~\eqref{eq:rho:scrr}, the ratio of the two bounds is $(\alpha_{\kk-1}^2 + \alpha_{\kk}^2)/(\alpha_{\kk} + \eps_k(\alpha;b) )^2 \to 1 + (\alpha_{\kk-1}/\alpha_\kk)^2$ as $b \to 0$. This ratio could be arbitrarily large depending on the relative sizes of $\alpha_{\kk}$ and $\alpha_{\kk-1}$. Thus, when the bounds give an accurate estimate of the misclassification rates of \scone and \scrr, we observe that \scrr has a clear advantage. This is empirically verified in Figure~\ref{fig:scrr:scone:boost}, for moderately dense cases. (In the very sparse case, the difference is not very much empirically.) In general, we expect \scrr to perform better when there is a large gap between $\sigma_\kk$ and $\sigma_{\kk-1}$, the two smallest nonzero singular values of $\Bb$.
	\end{exa}

	\begin{exa}[Rank-deficient connectivity]
		 Consider an extreme case where $\Psi$ is rank one: $\Psi = u v^T$ for some $u,v \in \reals_+^{k}$, where again for simplicity we have assumed $k_1 = k_2 = k > 1$. Also assume $n_1 \asymp n_2 =n$ and $\pi_{rj} = 1/k$ for $j \in [k]$ and $r=1,2$, i.e., the clusters are balanced. In this case, $\Bb = \Psi/k$ and $\sigma_k = 0$, hence Theorem~\ref{thm:scone} does not provide any guarantees for \scone. However, Theorem~\ref{alg:scrr} is still valid. We have $	\kk \Psinft[1]^2 = \Psinf[1]^2 = \kk^{-1} \norm{v}_2^2 \min_{s \neq t} (u_s - u_t)^2$ and $d \lesssim \infnorm{\Psi} \le \infnorm{u}\infnorm{v}$. It follows from Theorem~\ref{alg:scrr} that \scrr has average misclassification rate bounded as
		 \begin{align*}
		 	O \Big( \frac{k^2 \infnorm{u}\infnorm{v} }{ \norm{v}_2^2 \min_{s \neq t} (u_s - u_t)^2}\Big)
		 \end{align*}
		 whenever $k$ times the above is sufficiently small. This is a consistency result assuming that the coordinates of $u$ are different, all the elements of $u$ and $v$ are growing at the same rate and $k = O(1)$. 
	\end{exa}
\subsection{Efficient reduced-rank SC}
 The \scrr algorithm discussed above has the disadvantage of running a \kmeans algorithm on vectors in $\reals^n$ (the rows of $\Arek$, or in the ideal case the rows of $P$). We now introduce a variant of this algorithm that has the same performance as \scrr in terms of misclassification rate, while computationally is as efficient as \scone. This approach which we call efficient reduced-rank spectral clustering, \scerr, is detailed in Algorithm~\ref{alg:scerr}. The efficiency comes from running the \kmeans step on vectors in $\reals^\kk$ which is usually a much smaller space than $\reals^n$ ($\kk \ll n$ in applications).

\begin{algorithm}[t]\label{alg:scerr}
	\caption{\scerr}
	\setstretch{1.2}
	\begin{algorithmic}[1]
		\medskip
		\State Apply degree regularization Algorithm~\ref{alg:deg:reg} to $A$ to obtain $\Are$.
		\State Obtain $\Arek = \Zh_1 \Sigh \Zh_2^T$, the $\kk$-truncated SVD of $\Are$.
		\State Output $\kalg(\Zh_1 \Sigh)$ where $\kalg$ is an isometry-invariant $\kappa$-approximate \kmeans algorithm.
	\end{algorithmic}
	\label{alg:scerr}
\end{algorithm}

For the \kmeans step in \scerr, we need a \kmeans (type) algorithm $\kalg$ that only uses the pairwise distances between the data points. We call such \kmeans algorithms \emph{isometry-invariant}:
\begin{defn}
	A \kmeans (type) algorithm $\kalg$ is isometry-invariant if for any two  matrices $X^{(r)} \in \rmat(n,d_r), r =1,2$, with the same pairwise distances among points|i.e., $\dr(x^{(1)}_i, x^{(1)}_j) = \dr(x^{(2)}_i, x^{(2)}_j)$ for all distinct $i,j \in [n]$, where $(x^{(r)}_i)^T$ is the $i$th row of $X^{(r)}$|one has
	\begin{align*}
		\Misb(\kalg(X^{(1)}), \kalg(X^{(2)})) = 0.
	\end{align*}
\end{defn} 
Although the rows of $\kalg(X^{(1)})$ and $\kalg(X^{(2)})$ lie in spaces of possibly different dimensions, it still makes sense to talk about their relative misclassification rate, since this quantity only depends on the membership information of the \kmeans matrices and not their center information. We have implicitly assumed that $\dr(\cdot,\cdot)$ defines a family of distances over all Euclidean spaces $\reals^{d}, d=1,2,\dots$. This is obviously true for the common choice $\dr(x,y) = \norm{x-y}_2$.  If algorithm $\kalg$ is randomized, we assume that the same source of randomness is used (e.g., the same random initialization) when applying to either of  the two cases $X^{(1)}$ and $X^{(2)}$. 

The following result guarantees that \scerr behaves the same as \scrr when one uses an isometry-invariant approximate \kmeans algorithm in the final step. 

\begin{thm}\label{thm:scerr}
	Consider the spectral algorithm \scerr given in Algorithm~\ref{alg:scerr}. 
	Assume that for a sufficiently small $C_1 > 0$, \eqref{eq:scrr:asump} holds.
	Then, under the SBM model of Section~\ref{sec:SBM}, w.h.p.,
	\begin{align*}
	\Misb(\kalg(\Zh_1 \Sigh),P) \; \le  \;C_1^{-1}\, (1+\kappa)^2 \Big(\frac{ k \dg}{\Psinf[1]^2}\Big).
	\end{align*}
\end{thm}

\begin{rem}[Comparison with existing results]\label{rem:comparison:consistent}
	The existing results hold under different assumptions. Table~\ref{tab:compare} provides a summary of some the recent results. The ``$p$ vs. $q$'' denotes the assortative case where the diagonal entries of $B$ are above $p$ and off-diagonal entries are below $q$.  The ``e.v. dep'' denotes whether the consistency result depends on the $k$th eigenvalue or singular value of $B$ (or $\Psi$). The ``\kmeans'' column records the dimension of the matrix on which a \kmeans algorithm is applied.
	 To allow for better comparison, let us consider a typical (special) case of the setting in this paper, where $n_1 = n_2$, $k_1=k_2$, $\Psi$ is symmetric and $\dg\asymp \Psinf[1]$. Then all the spectral methods in~Table~\ref{tab:compare} have misclassification rate guarantees that are polynomial in $d^{-1}$. General SBMs without assortative assumption (e.g., ``$p$ vs. $q$")  were considered in earlier literature \cite{rohe2011spectral, lei2015consistency}. However, the theoretical guarantees were provided for sufficiently dense networks. The generalization to sparse networks is considered in \cite{yun2014accurate, chin2015stochastic, gao2017achieving, gao2018community} using some regularization on the adjacency matrix. However, these results only apply to assortative networks and with extra assumptions. Overall, the algorithms \scone and \scerr require less assumptions than any existing works. We also note that the guarantees of Theorems~\ref{thm:scrr} and~\ref{thm:scerr} in the context of a general SBM are new and have not appeared before (not even in the dense case).
	\begin{table}[t]
		\caption{Comparison of consistency results (cf. Remark~\ref{rem:comparison:consistent}).}
		\begin{tabular}{c|c|c|c|c|c|c}
			& minimum degree &  $p$ vs. $q$ & $p=O(q)$& e.v. dep &  given $\Psi$ & $k$-means \\ \hline
			\cite{rohe2011spectral} &$\Omega(n/\sqrt{\log n}).$& No. & Not needed. & Yes. &  No. & $n\times k$. \\ \hline
			\cite{yun2014accurate} & Not needed. & Yes & Not needed. & No. & No. &$n\times n.$ \\ \hline
			\cite{lei2015consistency} & $\Omega(\log n).$ & No. & Not needed. & Yes. &  No. & $n\times k$. \\ \hline
			\cite{chin2015stochastic} & Not needed. & Yes. & Required. & Yes. &  Yes. & $n\times k$.  \\ \hline
			\cite{gao2017achieving} & Not needed. & Yes. & Required. & Yes. &  No. & $n\times k$.\\ \hline
			\cite{gao2018community} & Not needed. & Yes. & Required. & No. &  No. & $n\times n$.\\ \hline
			SC-1 & Not needed. & No. & Not needed. & Yes. &  No. & $n\times k$.\\ \hline
			SC-RRE & Not needed. & No. & Not needed. & No. &  No. & $n\times k$.\\ \hline
		\end{tabular}
		\label{tab:compare}
	\end{table}

\end{rem}

\subsection{Results in terms of mean parameters}\label{sec:res:mean:param}
One useful aspect of \scrre is that one can state its corresponding consistency result in terms of the \emph{mean parameters} of the block model.  Such results are useful when comparing to the optimal rates achievable in recovering the clusters. The row mean parameters of the SBM in Section~\ref{sec:SBM} are defined as $\Lambda_{s\ell} := B_{s\ell} \;n_{2\ell}$ for $(s,\ell) \in [\kk_1] \times [\kk_2]$ which we collect in a matrix $\Lambda = (\Lambda_{s\ell}) \in \reals^{\kk_1 \times \kk_2}$. To get an intuition for $\Lambda$ note that 
\begin{align*}
	\ex[A Z_2] = PZ_2 = Z_1 B N_2 = Z_1 \Lambda.
\end{align*}
Each row of $A Z_2$ is obtained by summing the corresponding row of $A$ over each of the column clusters to get a $\kk_2$ vector. In other words, the rows of $A Z_2$ are the sufficient statistics for estimating the row clusters, had we known the true column clusters. Note that $\ex [(A Z_2)_{i*}] = z_{1i}^T \Lambda$, where the notation $(\cdot)_{i*}$ denotes the $i$th row of a matrix. In other words, we have $\ex [(A Z_2)_{i*}] = \Lambda_{s*}$ if node $i$ belongs to row cluster $s$. Let us define the minimum separation among these row mean parameters:
\begin{align}\label{eq:Laminf:def}
	\Laminf^2 :=  \min_{t \neq s} \norm{\Lambda_{s*} - \Lambda_{t*}}^2.
\end{align}
We have the following corollary of Theorem~\ref{thm:scrr} which is proved in Appendix~\ref{sec:proof:consist:res}.

\begin{cor}\label{cor:mean:param}
	Assume that $\pinf{r} := \min_{t} \pi_{rt} \ge (\beta_r \kk_r)^{-1}$ for $r=1,2$, and let $\kk = \min\{\kk_1,\kk_2\}$ and $\alpha = n_2/n_1$.
	Consider the spectral algorithm \scrr given in Algorithm~\ref{alg:scrr}. Assume that for a sufficiently small $C_1 > 0$,
	\begin{align}\label{eq:scrr:asump:Lambda}
		\beta_1 \beta_2 \,\kk\, \kk_1 \kk_2 \,\alpha \frac{\infnorm{\Lambda}}{\Laminf^2} \le C_1 (1+\kappa)^{-2}.
	\end{align}
	Then, under the SBM model of Section~\ref{sec:SBM}, w.h.p.,
	\begin{align*}
	\Misb(\Pc_\kappa(\Arek),P) \; \le  \;C_1^{-1}\, (1+\kappa)^2
	\beta_2 \,\kk \,\kk_2 \, \alpha \frac{\infnorm{\Lambda}} {\Laminf^{2}} .
	\end{align*}
	
\end{cor}

We note that the exact same result as Corollary~\ref{cor:mean:param} holds for \scerr assuming the \kmeans step uses an isometry invariant algorithm as discussed in Section~\ref{sec:consist:res}. We refer to~\cite{pl-bipartite} for an application of this result in constructing optimal clusterings in the bipartite setting.

\section{Extensions}\label{sec:exten}

\subsection{Clusters on one side only}\label{sec:one:cluster}
The bipartite setting allows for the case where only one side has clusters. Assume that $A \sim \text{Ber}(P)$ in the sense of~\eqref{eq:matrix:ber:defn} and, for example, only side~2 has $k_2$ clusters. Then we can model the problem as $P$ having $k_2$ distinct columns. However, within columns we do not require any block constant structure, i.e., the $k_2$ distinct columns of $P$ are general vectors in $[0,1]^{n_1}$. This problem can be considered a special case of the SBM model discussed in Section~\ref{sec:SBM} where $k_1 = n_1$: We recall that $P = \Zb_1 \Bb \Zb_2^T$ where $\Zb_1$ is an orthogonal matrix of dimension $n_1 \times k_1$, hence $\Zb_1 = I_{n_1}$. All the consistency results of the paper thus hold, where we set $k_1 = n_1$. 


\subsection{More clusters than rank}\label{sec:more:clust:than:rank}
One of the unique features of the bipartite setting relative to the symmetric one is the possibility of having more clusters on one side of the network than the rank of the connectivity matrix. In the notation established so far, this is equivalent to $\kk_1 > \kk = \min\{\kk_1,\kk_2\}$. 
Let us first examine the performance of \scone. Recall the SVD of $\Bb = \Usi \Sigma \Vsi^T$, as given in Lemma~\ref{lem:P:svd}. In contrast to the case $\kk_1 = \kk$, where the singular vector matrix $\Usi$ has no effect on the results (cf. Theorem~\ref{thm:scone}), in the case $\kk_1 > \kk$, these singular vectors play a role. Recall that $\Usi$ is a $\kk_1 \times \kk$ orthogonal matrix, i.e., $\Usi \in \ort{\kk_1}{\kk}$. For a matrix $\kk_1 \times \kk_1$ matrix $M$, and index set $\Ic \subset [\kk_1]$, let $M_{\Ic}$ be the principal sub-matrix of $M$ on indices $\Ic \times \Ic$.
 We assume the following \emph{incoherence} condition:
\begin{align}\label{eq:U:incoh}
	\max_{\Ic \,\subset\,[\kk_1]:\; |\Ic|=2}\opnorm{(\Usi \Usi^T - I_{\kk_1})_{\Ic}}  \;\le\; 1-\rho_1,
\end{align}
for some $\rho_1 \in (0,1]$. 
%
Letting  $\usi_{s}^T$ be the $s$th row of $\Usi$, for $s \in [\kk_1]$, we note that $(\Usi \Usi^T)_{st} =  \ip{\usi_s,\usi_t}$, that is, $\Usi \Usi^T$ is the Gram matrix of the vectors $\usi_s,\, s \in [\kk_1]$.
 We have the following extension of Theorem~\ref{thm:scone}:

\begin{thm}\label{thm:scone:unequal}
	Consider the spectral algorithm \scone given in Algorithm~\ref{alg:scone}. Assume that for a sufficiently small $C > 0$,
	\begin{align*}
	\kk \dg \, \sigma_\kk^{-2} \le C(1+\kappa)^{-2} \rho_1.
	\end{align*}
	Then, under the SBM model of Section~\ref{sec:SBM}, w.h.p.,
	\begin{align}\label{eq:scone:unequal}
	\Misb(\Pc_\kappa(\Zh_1),\Zb_1) \; \lesssim \; (1+\kappa)^2 
	\beta_1 \,\Big(\frac{\kk}{ \rho_1 \kk_1} \Big) \Big(\frac{\dg}{  \sigma_\kk^2}\Big) .
	\end{align}
\end{thm}

The theorem is proven in Appendix~\ref{sec:proof:thm:scone:unequal}. The factor $\kk/(\kk_1 \rho_1) = \kk_2 /(\kk_1 \rho_1)$ in~\eqref{eq:scone:unequal} is the price one pays for the asymmetry of the number of communities, when applying \scone. (Recall that $\kk := \min\{\kk_1,\kk_2\} = \kk_2$ by assumption.) Note that increasing $\kk_1$ decreases $\kk/\kk_1$, and at the same time, often increases $\rho_1$ since it is harder to have many nearly orthogonal unit vectors in low dimensions. 

\begin{rem}\label{rem:scrr:more:than:rank}
	It is interesting to note that in contrast to \scone, the consistency results for~\scrre do not need any modification for the case where the number of clusters is larger than the rank. In other words, the same Theorems~\ref{thm:scrr} and~\ref{thm:scerr} hold regardless of whether $\kk_1 = \kk$ or $\kk_1 > \kk$, though the difficulty of the latter case will be reflected implicitly via a reduction in $\Psinf[1]^2$ and $\Psinft[1]^2$.
\end{rem}

\subsection{General sub-Gaussian case}\label{sec:gen:subg}
The analysis presented so far for network clustering problems can be extended to general sub-Gaussian similarity matrices. Consider a random matrix $A$ with block constant mean
\begin{align}
P := \ex[A] = Z_1 B Z_2^T, 
\end{align}
as defined in \eqref{eq:nonsym:mean:def}, and where $Z_r, r = 1,2$ are again membership matrices. However, here $A$ is not necessarily an adjacency matrix. We assume  that $A_{ij}$ are sub-Gaussian random variables independent across $(i,j)\in[n_1]\times[n_2]$ and let  $\sigma := \max_{i,j} \|A_{ij}-\ex A_{ij}\|_{\psi_2}$.
We recall that a univariate random variable $X$ is called sub-Gaussian if its sub-Gaussian norm is finite~\cite{vershynin2018high}:
	\begin{align}
	\|X\|_{\psi_2} := \inf\{ t>0: \ex \exp(X^2/t^2)\le 2\} < \infty
	\end{align}
%
Note that we do not assume $A_{ij}$ and $A_{i'j'}$ to have the same distribution or the same sub-Gaussian norm even if $Z_{1i} = Z_{1i'}$ and $Z_{2j} = Z_{2j'}$. 

\begin{algorithm}[t]
	\caption{\scerr for sub-Gaussian noise}
	\setstretch{1.2}
	\begin{algorithmic}[1]
		\medskip
		\State Obtain $A^{(k)} = \Zh_1 \Sigh \Zh_2^T$, the $\kk$-truncated SVD of $A$.
		\State Output $\kalg(\Zh_1 \Sigh)$ where $\kalg$ is an isometry-invariant $\kappa$-approximate \kmeans algorithm.
	\end{algorithmic}
	\label{alg:scsg}
\end{algorithm}

Adapting Algorithm~\ref{alg:scerr} to the general sub-Gaussian case, we have Algorithm~\ref{alg:scsg} with the following performance guarantee: 

\begin{thm}\label{thm:scerr:subg:noise}
	Let $\Binf[1]$ and $\Binft[1]$ be defined as in~\eqref{eq:Psinf:defs} with $\Psi$ replaced with $B$ and let 
	\[\bar n:= (n_1^{-1} + n_2^{-1})^{-1}.\]
	Consider the spectral algorithm for sub-Gaussian noise given in Algorithm~\ref{alg:scsg}. Assume that for a sufficiently small $C_1 > 0$, 
	\begin{align}
	\frac{\sigma^2}{\bar n} k \,\Binft[1]^{-2} \le C_1 (1+\kappa)^{-2}.
	\end{align}
	Then, under the model defined in this section, 
	\begin{align*}
	\Misb(\kalg(\Zh_1 \Sigh),P) \; \le  \;C_1^{-1}\, (1+\kappa)^2 \Big(\frac{\sigma^2}{\bar n}\frac{k}{ \Binf[1]^2}\Big).
	\end{align*}
	with probability at least $1-2e^{-(n_1+n_2)}$. 
\end{thm}

See Appendix~\ref{sec:proof:subg} for the proof.
Consider the typical case where the number of clusters $k$ and the separation between them, $\Binf[1]^2$, remains fixed as $\bar n \to \infty$. Then, $\sigma^2/ \bar n$ plays the role of the signal-to-noise ratio (SNR) for the clustering problem and as long as $\sigma^2/ \bar n = o(1)$,  Algorithm~\ref{alg:scsg} is consistent in recovering the clusters. Of course, there are other conditions under which we have consistency, e.g., cluster separation quantity $\Binf[1]^2$ could go to zero, and $k$ could grow as well (as $\bar n \to \infty$) and as long as $\sigma^2 / \bar n  = o( \Binf[1]^2/k )$ the algorithm remains consistent.



\subsection{Inhomogeneous random graphs}\label{sec:inhom:graphs}

Up to now, we have stated consistency results for the SBM of Section~\ref{sec:SBM}. SBM is often criticized for having constant expected degree for nodes in the same community, in contrast to degree variation observed in real networks. (Although, SBM in the sparse regime $d = O(1)$ can still exhibit degree variation within a community, since the node degrees will be roughly $\text{Poi}(d)$ distributed, showing little concentration around their expectations unless $d \to \infty$.) The popular remedy is to look at the degree-corrected block model (DC-SBM)~\cite{Karrer2011,zhao2012consistency,gao2018community}.  Instead, we consider the more general inhomogeneous random graph model (IRGM)~\cite{soderberg2002general,bollobas2007phase} which might be more natural in practice, since it does not impose the somewhat parametric restrictions of DC-SBM on the mean matrix. 

We argue that consistency results for the spectral clustering can be extended to a general inhomogeneous random graph (IRGM) model $A \sim \ber(P)$, assuming that the mean matrix can be well-approximated by a block structure. Let us consider a scaling as before:
\begin{align}\label{eq:irgm:mean:a}
	P := \frac{\Pn}{\sqrt{n_1 n_2}}, \quad\text{and take}\quad d = \sqrt{\frac{n_2}{n_1}} \infnorm{\Pn}.
\end{align}
Note that $P \in [0,1]^{n_1 \times n_2}$ is a not assumed to have any block structure. However, we assume that there are membership matrices $Z_r \in \hard_{n_r,\kk_r}, \, r=1,2$ such that $P$ over the blocks defined by $Z_1$ and $Z_2$ is approximately constant.  Let
\begin{align}\label{eq:Bt:def}
	\Bt_{st} := \frac1{n_{1s} n_{2t}} \sum_{ij} P_{ij} (Z_1)_{is} (Z_2)_{jt}, \quad s \in [\kk_1],\; t \in [\kk_2]
\end{align}
be the mean (or average) of $P$ over these blocks, and let $\Bt = (\Bt_{st}) \in [0,1]^{\kk_1 \times \kk_2}$. Compactly, $\Bt = N_1^{-1} Z_1^T P Z_2 N_2^{-1}$, in the notation established in Section~\ref{sec:SBM}.  We can define the SBM approximation of $P$ as
\begin{align}\label{eq:sbm:approx}
	\Pt := Z_1 \Bt Z_2^T = \Zb_1 \Zb_1^T P \Zb_2 \Zb_2^T = \PZ_1 P \PZ_2,
\end{align}
recalling $\Zb_r = Z_r N_r^{-1/2}$ and introducing the notation $\PZ_r = \Zb_r \Zb_r^T$. Note that  $\PZ_r$ is a rank $\kk_r$ projection matrix. According to~\eqref{eq:sbm:approx}, the map $P \mapsto \Pi_1 P \Pi_2$ takes a (mean) matrix to its SBM approximation relative to $Z_1$ and $Z_2$. We can thus define a similar approximation to $\Pn$,
\begin{align*}
	\Pnt := \Pi_1 \Pn \Pi_2,  \quad \text{noting that}\quad \Pt = \frac{\Pnt}{\sqrt{n_1 n_2}}.
\end{align*}

In order to retain the qualitative nature of the consistency results, the deviation of each entry $\Pn_{ij}$ from its block mean $(\Pnt)_{ij}$ should not be much larger than $(\Pnt)_{ij}^{1/2}$. In fact, we allow for a potentially larger deviation, assuming that:
\begin{align}\label{asu:Pn:dev}
	\frac{1}{\sqrt{n_1 n_2}}\fnorm{\Pn - \Pnt} \le  C_0 \sqrt{d},
\end{align}
for some constant $C_0 > 0$ and $a$ satisfying~\eqref{eq:irgm:mean:a}. The key is the following concentration result:
\begin{prop}\label{prop:irgm:concent}
	Assume that $A \sim \ber(P)$ as in~\eqref{eq:matrix:ber:defn}, with $n_1 \le n_2$, and let $\Are$ be obtained from $A$ by the regularization procedure in Theorem~\ref{thm:concent:nonsym}, with $d' = d$ as given in~\eqref{eq:irgm:mean:a}. Let $\Pt$ be as defined in~\eqref{eq:sbm:approx}, and assume that it  satisfies~\eqref{asu:Pn:dev}. Then, with probability at least $1-n_2^{-c}$, 
	\begin{align}\label{eq:irgm:concent}
		\opnorm{\Are - \Pt } \le c_2 \sqrt{\dg}.
	\end{align}
\end{prop}

Proposition~\ref{prop:irgm:concent} provides the necessary concentration bound required in Step~1 of the analysis outlined in Section~\ref{sec:analysis:sketch}. In other words, bound~\eqref{eq:irgm:concent} replaces \eqref{eq:gen:concent} by guaranteeing a similar order of deviation for $\Are$ around $\Pt$ instead of $P$. Thus, all the results of the paper follow under an IRGM with a mean matrix satisfying~\eqref{asu:Pn:dev}, with the modification that the connectivity matrix and all the related quantities (such as $\sigma_\kk, \Psinf[1]^2$ and so on) are now based on $\Bt$, the connectivity matrix of the corresponding approximate SBM, as given in~\eqref{eq:Bt:def}.

\subsubsection{Graphon clustering}\label{sec:graphon:clust}
	As a concrete example of the application of Proposition~\ref{prop:irgm:concent}, let us consider a problem which we refer to as \emph{graphon clustering}.
	For simplicity, consider the case $n_1 = n_2 = n$.
	Let $\rhoz: [0,1]^2\to \reals_+$ be a bounded measurable function, and let $X_1, \dots, X_n, Y_1, \dots, Y_n\sim \text{Unif}([0,1])$ be an i.i.d. sample. Assume that 
	\begin{align}\label{eq:graphon:model}
		A_{ij} \mid \{X_{i'}\}, \{Y_{j'}\}  \;\sim\; \text{Bern}\big(\, \rhoz(X_i, Y_j)/n \,\big),
	\end{align}
	independently over $i,j= 1,\dots,n$.
	 Consider two partitions of the unit interval:  $\biguplus_{s=1}^{k_1} I_s=\biguplus_{t=1}^{k_2} J_t=[0,1]$ where $I_s$ and $J_t$ are (unknown) measurable sets. Let $\Psi \in [0,1]^{k_1\times k_2}$ and consider the block-constant function $\rhozt$ on $[0,1]^2$ defined by
	\begin{align}\label{eq:blk:const:rhozt}
	\rhozt:=\sum_{s=1}^{k_1}\sum_{t=1}^{k_2} \Psi_{st}1_{I_s\times J_t},
	\end{align}
	where $1_{I_s\times J_t}(x,y) = 1_{I_s}(x) 1_{J_t}(y)$ is the indicator of the set $I_s \times J_t$. Whenever $\rhoz$ has a block constant approximation such as $\rhozt$, we can consider $I_s$ and $J_t$ as defining implicit (true) clusters over nodes. More precisely, row node $i$ belongs to community $s$ if $1\{X_i \in I_s\}$ and similarly for the column labels. Note that none of $\{X_i\}$, $\{Y_j\}$, $\{I_s\}$ and $\{J_t\}$ are observed. One can ask whether we can still recover these implicit clusters given an instance of $A$.
	
		\begin{prop}\label{prop:rho:dev}
		Assume that there exists a function $\rhozt$ of the form~\eqref{eq:blk:const:rhozt}, such that
		\begin{align}\label{eq:L4:bound}
			\|\rhozt-\rhoz\|_{L^4}=o(\sqrt d)
		\end{align}
		Consider normalized mean matrices $\Pn$ and $\Pnt$ with entries, $\Pn_{ij}=\rhoz(X_i,Y_j)$ and $\Pnt_{ij}=\rhozt(X_i,Y_j)$.  
		Then, with high probability  $\fnorm{\Pn-\Pnt} \le n\sqrt \dg$.
		
	\end{prop}

Proposition~\ref{prop:rho:dev} shows that a fourth moment bound of the form~\eqref{eq:L4:bound} is enough to guarantee~\eqref{asu:Pn:dev} and as a consequence the result of Proposition~\ref{prop:irgm:concent}. In order to apply the results of the paper, we require that~\eqref{eq:L4:bound} holds with $d = \infnorm{\rhoz}$. Then, all the consistency results of the paper follow with $\Psi$ replaced by that from~\eqref{eq:blk:const:rhozt}, and regularization Algorithm~\ref{alg:deg:reg} replaced with that of Theorem~\ref{thm:concent:nonsym}. 

For example, with $\sigma_k$  as defined in Section~\ref{sec:SBM} (using $\Psi$ from~\eqref{eq:blk:const:rhozt} in forming $\Bb$), Theorem~\ref{thm:scone} gives a misclassification rate  at most $O( \dg/\sigma_k^2)$ for recovering the labels $1\{X_i \in I_s\}$ by the \scone algorithm. In typical cases, it is plausible to have $\sigma_k^2 \asymp \dg^2$ (cf.~Remark~\ref{rem:why:it:is:consistency}), hence a misclassification rate of  $O(1/\dg) = o(1)$ as $\dg \to \infty$. Proposition~\ref{prop:rho:dev} is an illustrative example, and one can obtain other conditions by assuming more about the deviation $\rhozt-\rhoz$, such as boundedness.

\begin{rem}
	We note that graphon clustering problem considered above is different from what is typically called graphon estimation. In the latter problem, under a model of the form~\eqref{eq:graphon:model}, one is interested in recovering the mean matrix $\ex[A]$ or $\rhoz$ in the MSE sense. This problem has been studied extensively in recent years, often under the assumption of smoothness of $\rhoz$, for maximum likelihood SBM approximation~\cite{airoldi2013stochastic, olhede2014network, gao2015rate, klopp2017oracle}  and spectral truncation~\cite{xu2017rates}. The graphon clustering problem, as far as we know, has not been considered before and is concerned with recovering the underlying clusters, assuming that such \emph{true clusters exist}. We only need the existence of a block constant approximation $\rhozt$, over the true clusters (of the form~\eqref{eq:blk:const:rhozt}) that satisfies $\|\rhozt-\rhoz\|_{L^4}=o(\sqrt d)$ for $d = \infnorm{\rhozt}$. Then our results implicitly imply that the underlying clusters are identifiable and consistently recovered by spectral approaches. Note that we do not impose any explicit smoothness assumption on $\rho$ and there is no lower bound requirement on $d$ (such as $d = \Omega(\log n)$ in~\cite{xu2017rates}).
\end{rem}

\section{Simulations}\label{sec:sim}

We now present some simulation results showing the performance of the data-driven regularization of Section~\ref{sec:data:driven:trunc}.
We sample from the bipartite SBM model with connectivity matrix
\begin{align}\label{eq:sim:B}
B = \sqrt{ \frac{\log(n_1 n_2)}{n_1 n_2} } B_0, \quad \text{ where }\quad
B_0 = \frac12
\begin{bmatrix}
6 & 1 & 1 & 1 \\
1 & 6 & 1 & 1 \\
1 & 1 & 6 & 1
\end{bmatrix},
\end{align}
for which $\Psi = \sqrt{\log (n_1n_2)} B_0$ as in~\eqref{eq:nonsym:mean:def}.
We let $n_1 = n_0 k_1$ and $n_2 = n_0 k_2$, and we vary $n_0$. Note that $k_1 = 3$ and $k_2 = 4$. We measure the performance using the normalized mutual information (NMI) between the true and estimated clusters. The NMI belongs to the interval $[0,1]$ and is monotonically increasing with clustering accuracy.  We consider Algorithm~\ref{alg:deg:reg} with regularization parameter $\tau = 1, 1.2, 1.4$ and $\infty$, where $\tau=\infty$ corresponds to no regularization. Although we have established theoretical guarantees for  $\tau = 3$, this constant is not optimal and any scalar $\asymp$ 1 might perform well.

In this model, the key parameter $d =  \sqrt{n_2/n_1} \infnorm{\Psi} \sim \sqrt{\log n_0}$ as $n_0 \to \infty$. In other words, the maximum expected degree of the network scales as $\sqrt{\log n_0}$ which is enough for the consistency of spectral clustering; in fact, results in Section~\ref{sec:consist:res} predict  a misclassification rate of $O( d^{-1} ) = O\big( (\log n_0)^{-1/2}\big)$ for various spectral algorithms discussed in this paper.

\begin{figure}[t]
	\begin{tabular}{cc}
		\includegraphics[width=.49\textwidth]{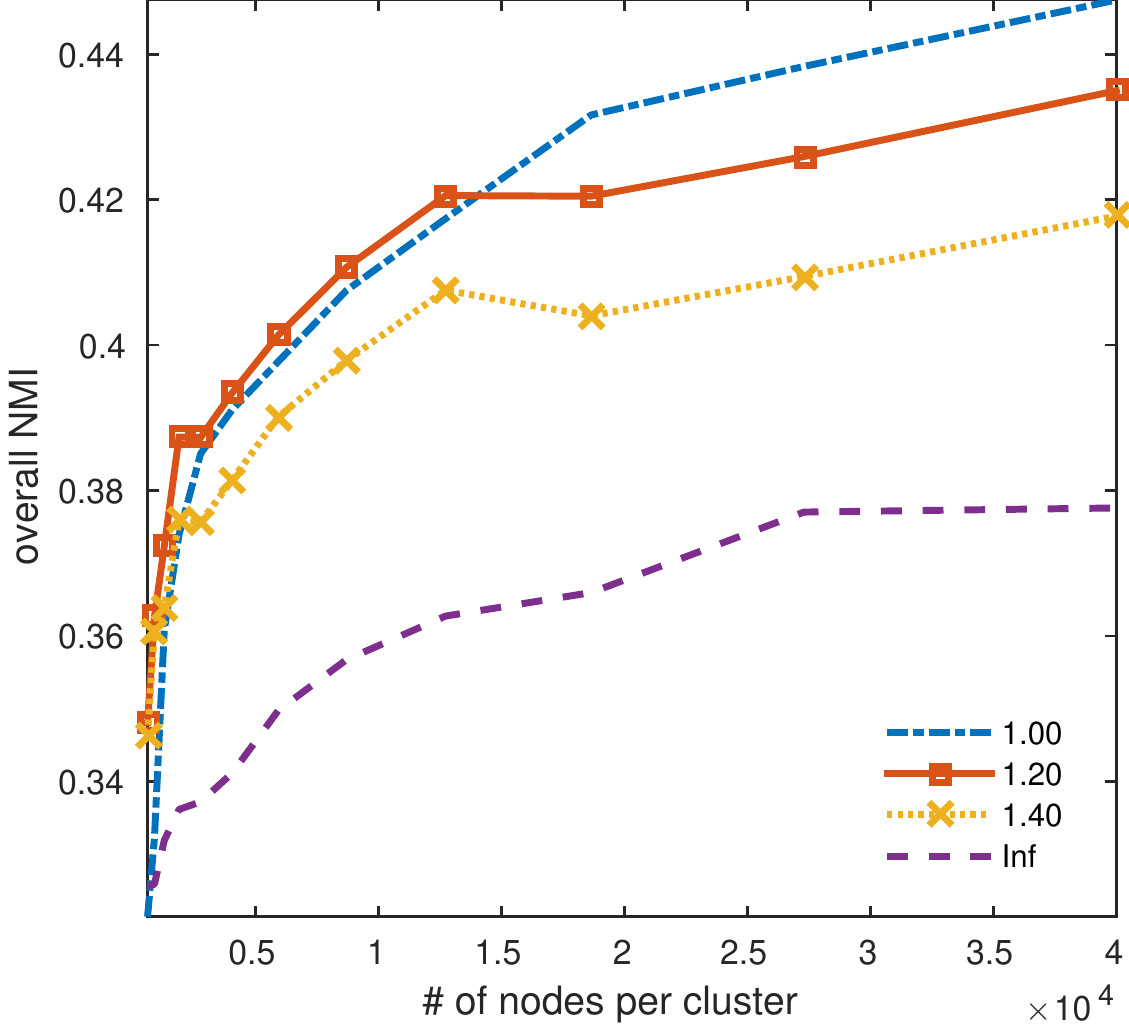}\; &
		\includegraphics[width=.49\textwidth]{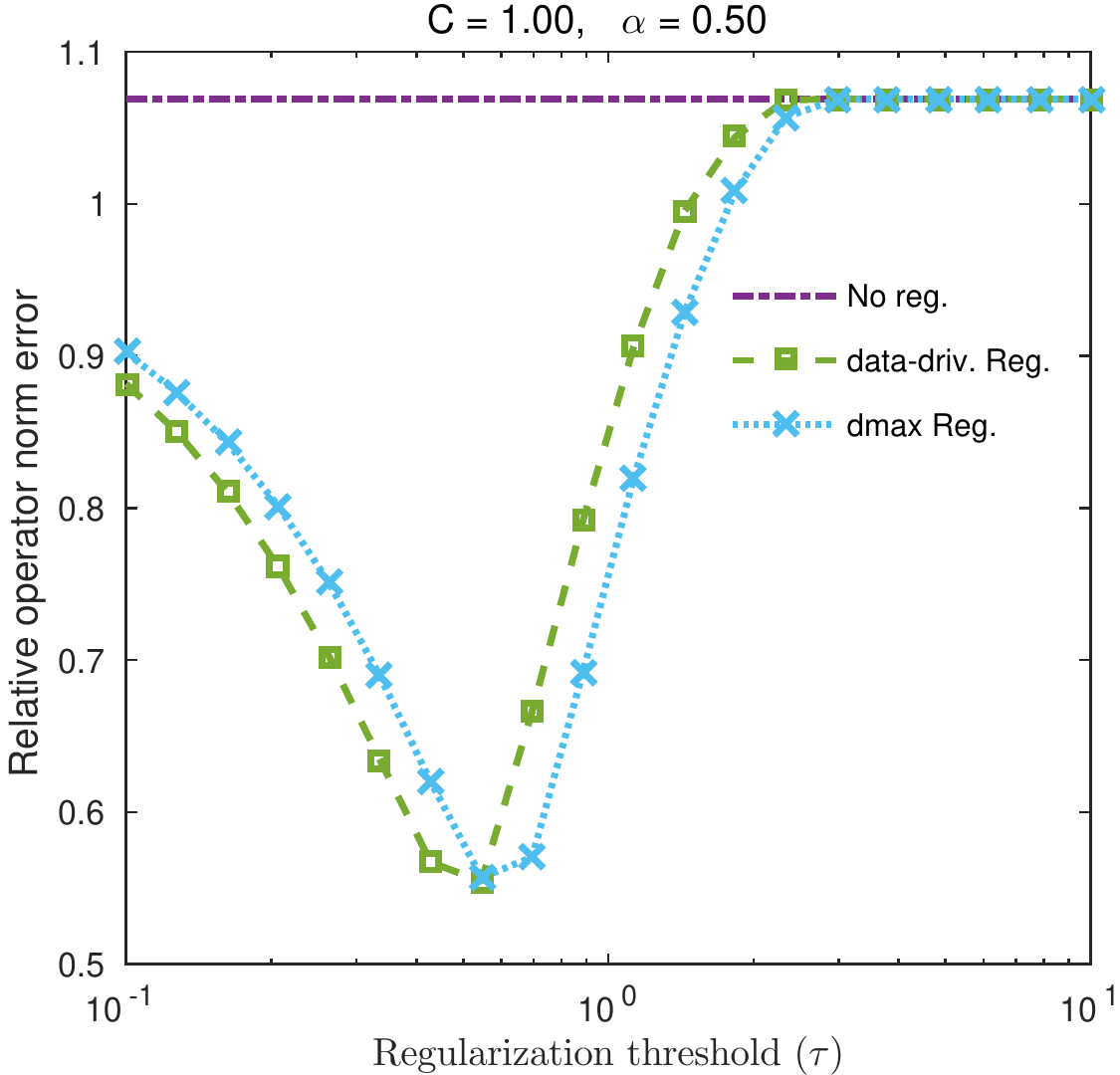} \\
		(a) & (b)
	\end{tabular}
	\caption{(a) NMI plots for  \scerr algorithm with degree regularization Algorithm~\ref{alg:deg:reg} for various values of $\tau = 1,1.2,1.4$ and $\infty$. (b) Relative operator norm error between the adjacency matrix and its expectation, with data-driven and oracle ($\dmax$) regularization  as well as no regularization.}
	\label{fig:nmi}
\end{figure}

Figure~\ref{fig:nmi}(a) shows the NMI plots as a function of $n_0$ for the \scerr algorithm with the regularization scheme of Algorithm~\ref{alg:deg:reg}. In the \kmeans step, we have used \verb|kmeans++| which as described in Remark~\ref{rem:kmeans} satisfies the approximation property of Section~\ref{sec:kmeans:step} with $\kappa = O(\log (k_1 \wedge k_2)) = O(1)$ in this case. The results are averaged over $15$ replicates. The plots clearly show that  the regularization considerably boosts the performance of spectral clustering for model~\eqref{eq:sim:B}. 

Figure~\ref{fig:nmi}(b) shows the relative operator norm error  between the (regularized) adjacency matrix and its expectation, i.e., $\opnorm{\Are - \ex A } / \opnorm{\ex A}$ with and without regularization ($\Are = A$). The plots correspond to the same SBM model with $n_0 = 500$ (and the results are averaged over 3 replicates). For the regularization, we consider both the oracle where the degrees are truncated to $\tau \dmax$ and $\tau \dmax'$ for the rows and columns (see Section~\ref{sec:data:driven:trunc} for the definitions of these quantities) as well as the data-driven one that truncates as in Algorithm~\ref{alg:deg:reg}. The plots show the relative error as a function of $\tau$ and we see that the regularization clearly improves the concentration. We also note that the behavior of the data-driven truncation closely follows that of the oracle as predicted by Theorem~\ref{thm:concent:data:driven}. 
\section*{Acknowledgement}
We thank Zahra S. Razaee, Jiayin Guo and Yunfeng Zhang for helpful discussions.

\printbibliography


\appendix

\section{Proofs}
\subsection{Proofs of Section~\ref{sec:dilation}}
\begin{proof}[Proof of Lemma~\ref{lem:DK:Z:dev}]
	Let $\Wb$ and $\Wh$ be the $W$ of Lemma~\ref{lem:dilation}(a) for $P^\dagger$ and $\Are^\dagger$, respectively. Let us also write $\Wb_1$ and $\Wh_1$ for the $(n_1+n_2) \times \kk$ matrices obtained by taking the submatrices of $\Wb$ and $\Wh$ on columns $1,\dots,k$. We have
	\begin{align*}
	\Wb_1 = \frac1{\sqrt{2}}
	\begin{pmatrix}
	\Zb_1 \Usi \\
	\Zb_2 \Vsi
	\end{pmatrix}, \quad 
	\Wh_1 = \frac1{\sqrt{2}}
	\begin{pmatrix}
	\Zh_1  \\
	\Zh_2 
	\end{pmatrix}.
	\end{align*}
	Note that $\Wb_1,\Wh_1 \in \ort{(n_1 +n_2)}{k}$.  Let $\proj_{\Wb_1}$ be the (orthogonal) projection  operator, projecting onto $\img(\Wb_1)$, i.e., the column span of $\Wb_1$, and similarly for $\proj_{\Wh_1}$. We have
	\begin{align*}
	\opnorm{\Pi_{\Wh_1}  - \Pi_{\Wb_1} } 
	&\le \frac{2}{2 \sigma_\kk}\opnorm{\Are^\dagger - P^\dagger} 
	&& (\text{Symmetric DK and Lemma~\ref{lem:dilation}(d)})  \\
	&= \frac{1}{ \sigma_K}\opnorm{(\Are - P)^\dagger}
	&& (\text{Linearity of dilation}) \\
	&= \frac{1}{ \sigma_K}\opnorm{\Are - P}
	&& (\text{Lemma~\ref{lem:dilation}(b)}).
	\end{align*}
	The next step is to translate the operator norm bound on spectral projections into a Frobenius bound. The key here is the bound on the rank of spectral deviations which leads to a $\sqrt{k}$ scaling as opposed to $\sqrt{n_1 + n_2}$, when translating from operator norm to Frobenius: 
	\begin{align*}
	\min_{Q\, \in \ort{k}{k}}\; \fnorm{\Wh_1-\Wb_1 Q}&\le \fnorm{\Pi_{\Wh_1}  - \Pi_{\Wb_1} } &&\text{(By Lemma~\ref{lem:proj:align} in Appendix~\ref{app:aux})}\\
	&\le \sqrt{2\kk}\, \opnorm{\Pi_{\Wh_1}  - \Pi_{\Wb_1} } && (\,\rank(\Pi_{\Wh_1}  - \Pi_{\Wb_1} ) \le 2\kk\,)\\
	&\le \frac{\sqrt{2\kk}}{\sigma_\kk} \opnorm{\Are - P}.
	\end{align*}
	Since 
	$2\fnorm{\Wh_1-\Wb_1 Q}^2 = \fnorm{\Zh_1 - \Zb_1 \Usi Q}^2 +  \fnorm{\Zh_2 - \Zb_2\Vsi Q}^2$, we obtain the desired result after combining with~\eqref{eq:gen:concent}.
\end{proof}

\subsection{Proofs of Section~\ref{sec:concent}}\label{sec:proof:deg:reg}
Let us start with a relatively well-known concentration inequality:
\begin{prop}[Prokhorov]\label{prop:prokh:concent}
	Let $S = \sum_{i} X_i$ for independent centered variables $\{X_i\}$, each bounded by $c < \infty$ in absolute value a.s. and suppose $v\ge \sum_i \ex X_i^2$, then
	\begin{align}\label{eq:poi:concent}
	\pr\big(S > vt \big) \le \exp[ {- v h_c(t)} ], \quad t \ge 0, \quad \text{where}\;
	h_c(t) := \frac3{4c} t \log \big(1+\frac{2c}{3} t\big).
	\end{align}
	Same bound holds for $\pr(S < -vt)$.
\end{prop}

We often apply this result with $c=1$. We note that for any $u \ge \frac23\alpha$ any $\alpha > 0$,
\begin{align}\label{eq:h:lower:bound}
h_1(\alpha u) \ge \frac{3}{4} \alpha  u \log \Big(\frac23 \alpha u\Big) \ge \frac{3}{4} \alpha u \log u^2 \ge \frac{3}{2} \alpha  u  \log u.
\end{align}


\begin{proof}[Proof of Lemma~\ref{lem:deg:trunc}]
		We first note that
	\begin{align}
	\pr( D_i - d_i > v t) &\le \exp(- v h_1(t)), \quad \forall v \ge d_i,\; t \ge 0, \label{eq:Di:upper}\\
	\pr( D_i - d_i < - v t)&\le \exp(- v h_1(t)), \quad \forall v \ge d_i,\; t \ge 0 \label{eq:Di:lower}.
	\end{align}

	\textbf{Lower bound.}
	Let $Y_i =  1\{D_i < \dmax/2 \}$. Without loss of generality, assume that cluster~1 achieves the maximum in the definition of $\dmax$ (i.e., $t^*=1$), and fix $i$ such that  $z_{1i} = 1$. We have $d_i = \dmax$, hence applying~\eqref{eq:Di:lower} with $v = d_i = \dmax$ and $t = 1/2$,
	\begin{align*}
	\pr(D_i < \dmax / 2)   \; \le\; \exp(-\dmax/10),
	\end{align*}
	using  $h_1(1/2)=0.2157>1/10$.
	Let $q := \exp(-\dmax/10)$ so that $\ex Y_i  \le q$. By assumption $\dmax \ge \db$ is sufficiently large that $q \le 1/4$. We then have
	\begin{align*}
	\pr(D_{(\floor{n_1/\Db})}< d_{\max}/2) 
	&\le \pr\Big(\sum_{i=1}^{n_1} 1\{D_i\ge d_{\max}/2\}\le n_1/\Db\Big) 
	\le \pr\Big(\sum_{i:z_{1i}=1} 1\{D_i\ge d_{\max}/2\}\le n_1/\Db\Big)\\
	&=\pr\Big(\sum_{i:z_{1i}=1}Y_i >  n_{11} - n_1/\Db\Big)\\
	&\le\pr\Big(\sum_{i:z_{1i}=1}Y_i > n_{11}- 2n_1/\db\Big)+\pr\Big( \Db <\frac{\db}2 \Big)\ =: T_1 + T_2
	\end{align*}
	For the first term we have
	\begin{align*}
	T_1 \le \pr\Big(\sum_{i:z_{1i}=1} (Y_i-\ex Y_i) >  n_{11}(1-q) - 2 n_1/\db
	\,\Big).
	\end{align*}
	By assumptions $n_{11} \ge  n_1/(\beta k_1)$ and $\beta k_1 / \db \le 1/8$, we have $n_{11}(1-q) - 2 n_1/\db \ge n_{11}/2$. Taking $c=1$, $v=n_{11}q$, $t=1/(2q)$ in Proposition~\ref{prop:prokh:concent}, we have $vt = n_{11}/2$ and $t \ge 2$
	\begin{align*}
	T_1 \le \pr\Big(\sum_{i:\,z_{1i}=1} (Y_i-\ex Y_i) > n_{11}/2\Big) 
	&\le \exp\Big({-} n_{11} \,q\, h_1\big(\frac1{2q}\big)\Big)  \\
	&\le \exp\big({-} c_1 \, n_{11} \dmax \big) \le e^{-3 n_1 /5}
	\end{align*}
	%
	where we have applied~\eqref{eq:h:lower:bound} with $\alpha=1/2$ and $u = 1/q = e^{\dmax/10}$ to get
	\begin{align*}
	q\, h_1\big(\frac1{2q}\big) 
	\ge q \frac{3}{2} \frac{1}{2q} \log (1/q) = \frac{3}{40} \dmax
	\end{align*}
	and $n_{11} \dmax \ge n_1 \dmax / (\beta k_1) \ge 8 n_1$ since $\beta k_1 / \dmax \le  \beta k_1 / \db \le 1/8$ by assumption.
	
	\medskip
	For the second term $T_2$, we note that $n_1 (\Db - \db) = \sum_{i,j} (A_{ij} - \ex A_{ij})$. Applying Proposition~\ref{prop:prokh:concent} with $v = n_1 \db$ and $t=1/2$, we have
	\begin{align}\label{eq:T2:bound}
	T_2 = \pr\Big( n_1 (\Db - \db) < -n_1 \db/2\Big) \le \exp(- n_1 \db /10) \le e^{-n_1/ 5}.
	\end{align}
	using $h_1(1/2) \ge 1/10$ and $\db \ge 2$.
	
	\paragraph{Upper bound.} For any $ \in [n_1]$, applying~\eqref{eq:Di:upper} with $v = \dmax \ge d_i$, and $t=1/2$, we have 
	\begin{align*}
	\pr(D_i > 3 \dmax / 2)   \; \le\; \exp(-\dmax/10).
	\end{align*}
	Let $Z_i =  1\{D_i > 3 \dmax/2 \}$ and note that $\ex Z_i \le q = \exp(-\dmax/10)$, as in the case of the lower bound. Then,
	\begin{align*}
	\pr\Big(D_{(\floor{n_1/\Db})} > \frac32d_{\max}\Big) 
	&\le \pr\Big(\sum_{i=1}^{n_1} Z_i\ge \frac{n_1}{\Db} -1\Big)\\
	&\le\pr\Big(\sum_{i=1}^{n_1}Z_i\ge \frac {2n_1}{3\db}-1\Big)
	+\pr\Big( \Db >\frac{3\db}2 \Big) =: T_1' + T_2'.
	\end{align*}
	Using assumption $n_1 / \db \ge 2$, we have $2 n_1/(3\db) - 1 \ge n_1/(6\db) \ge n_1 / (6 \dmax)$, hence
	\begin{align}\label{eq:Tpp1}
	T_1' \;\le\; \pr\Big(\sum_{i=1}^{n_1}(Z_i - \ex Z_i) \ge \frac {n_1}{6\dmax} - n_1 q\Big) =: T''_1
	\end{align}
	We claim that if $\dmax \ge 1$, then for any $\gamma \ge 1$,
	\begin{align}\label{eq:gamma:dev}
	f(\gamma) := \pr\Big(\sum_{i=1}^{n_1}(Z_i - \ex Z_i) \ge \frac{n_1}{\gamma \dmax}\Big) \le 
	e^{-3 n_1 / (20\gamma)}.
	\end{align}
	To see this, taking $c=1$, $v=n_{1}q$, $t=1/(\gamma q \dmax)$ in Proposition~\ref{prop:prokh:concent}, we have
	\begin{align*}
	f(\gamma) 
	&\le \exp\Big({-} n_1 q \,h_1 \Big( \frac{1}{\gamma q \dmax} \Big) \Big).
	\end{align*}
	Applying~\eqref{eq:h:lower:bound} with $\alpha = 1/(\gamma \dmax)$ and $u = 1/q = e^{\dmax/10}$, noting that $u \ge (2/3) \alpha$,
	\begin{align*}
	q\, h_1\big(\frac1{\gamma q \dmax}\big) \ge q \frac32 \frac{1}{\gamma q \dmax} \log (1/q) = \frac3{20 \gamma}
	\end{align*}
	showing~\eqref{eq:gamma:dev}. 
	
	Going back to bounding $T''_1$, for sufficiently large $\dmax$, we have $1/(6\dmax) - q \ge 1/(12\dmax)$. Hence, we can apply~\eqref{eq:gamma:dev} with $\gamma = 12$ to conclude,
	$	T''_1 \le e^{- n_1 /80}.$ 
	%
	For $T'_2$, by the same argument as in~\eqref{eq:T2:bound}, we have $T_2 \le e^{-n_1/5}$.

	\paragraph{Proof of part~(b).} From the lower bound, we have that 
	$3D_{(\floor{n_1/\Db})} > \frac32d_{\max}$ on an event $\Ac$  with probability $\pr(\Ac)\ge 1-2 e^{- n_1/80}$. 
	Then,
	\begin{align*}
	W_i:= 1\big\{ D_i>3D_{(\lfloor n_1/\Db\rfloor)} \big\} \;\le\;  1\{D_i > 3 \dmax/2 \} = Z_i, \quad \text{on $\Ac$}.
	\end{align*}
	Hence, $\big|i: D_i>3D_{(\lfloor n_1/\Db\rfloor)}\big| = \sum_i W_i$ and
	\begin{align*}
	\pr \Big( \Big\{\sum_{i=1}^{n_1} W_i > \frac{10 n_1}{d} \Big\} \cap \Ac \Big) &\le  \pr \Big( \sum_{i=1}^{n_1} Z_i > \frac{10 n_1}{d} \Big) 
	= \pr \Big( \sum_{i=1}^{n_1} (Z_i - \ex Z_i) > \frac{10n_1}{d}  - n_1 q \Big).
	\end{align*}
	Since $\dmax\ge n_2\|P\|_\infty/(\beta k_2)=d/(\beta k_2)$,
	and $\beta k_2 \le \db /8 \le \dmax/8$ we have $\dmax^2 \ge  8d $. Recalling that $q = e^{-\dmax/10}$
	\begin{align}\label{eq:temp:34}
	q d = (q \dmax^2) (d/\dmax^2) \le 1/8.
	\end{align}
	for $\dmax \ge 90$.	We have
	\begin{align*}
	\frac{10n_1}{d} - n_1 q \ge \frac{9n_1}{d} \ge \frac{9n_1}{\dmax} \frac{1}{\beta k_2}
	\end{align*}
	hence we can apply~\eqref{eq:gamma:dev} with $\gamma = \beta k_2/9$ to obtain
	\begin{align*}
	\pr \Big( \sum_{i=1}^{n_1} W_i > \frac{10 n_1}{d} \Big) \le \exp\Big( {-} \frac{27 n_1}{20 \beta k_2}\Big).
	\end{align*}
The proof is complete. 
\end{proof}

\begin{proof}[Proof of Lemma~\ref{lem:deg:trunc:col}]
	The proof of part~(a) follows exactly as in the case of row degrees establishing the result with probability at least $1 - 3 e^{-n_2/80}$. For part~(b), we have by the same argument as in the case of row degrees
	\begin{align*}
	\pr \Big( \big| j: D'_j>3D'_{(\lfloor n_2/\Db'\rfloor)} \big| > \frac{10n_2}{d}\Big)  
	&\le  \pr \Big( \sum_{j=1}^{n_2} (Z'_j - \ex Z_j') > \frac{10 n_2}{d} - n_2 q'\Big) + \pr (\Ac'^c) \\
	&:= T_3 + \pr (\Ac'^c)
	\end{align*}
	where $\Ac'$ is defined similar to $\Ac$ for column variables and is controlled similarly. We also have $q' := \exp(-\dmax'/10)$.
	
	To simplify notation, let $\alpha := n_2/n_1 \ge 1$.
	We have $\dmax' \ge n_1 \infnorm{P}/(\beta k_1) =  d /(\alpha \beta k_1)$, recalling $d = n_2 \infnorm{P}$. By assumption, $\db \ge 8\beta k_1 \alpha^2$. Since $n_1 \db = \sum_{ij} p_{ij} = n_2 \db'$, that is, $\db = \alpha \db'$, we obtain $\beta k_1 \le \db / ( 8 \alpha^2 ) =  \db'/ (8\alpha) \le  \dmax'/(8\alpha)$. It follows that $(\dmax')^2 \ge 8 d$ which is similar to what we had for the rows; hence,  $q'd \le 1/8$ as long as $\dmax' \ge 90$ (see~\eqref{eq:temp:34}), which is true since $\dmax' \ge \db' =\db / \alpha \ge \db / \alpha^2 \ge 90$ by assumption. Then,
	\begin{align*}
	\frac{10n_2}{d} - n_2 q' \ge \frac{9n_2}{d} \ge \frac{9n_2}{\dmax'} \frac{1}{ \alpha \beta k_1}.
	\end{align*}
	Therefore, applying the column counterpart of~\eqref{eq:gamma:dev} with $\gamma = \alpha \beta k_1 /9$,
	\begin{align*}
	T_3 \le  \exp\Big( {-} \frac{27 n_2}{20 \alpha \beta k_1}\Big) =  \exp\Big( {-} \frac{27 n_1}{20 \beta k_1}\Big).
	\end{align*}
	The proof is complete. 
\end{proof}	

\subsection{Proofs of Section~\ref{sec:kmeans:step}}

\begin{proof}[Proof of  Corollary~\ref{cor:kmeans:misclass}]
	Using~\eqref{eq:kmeans:proj:closeness}, we have $\df(\Xs,\Xt) \le (1+\kappa) \eps$ for any $\Xt \in \Pc_\kappa(\Xh)$. We now apply Proposition~\ref{prop:kmeans:misclass} to $\Xs$ and $\Xt$, both \kmeans matrices, with $(1+\kappa)\eps$ in place of $\eps$. 
\end{proof}

\begin{proof}[Proof of Proposition~\ref{prop:kmeans:misclass}]
	The proof follows the argument in~\cite[Lemma~5.3]{lei2015consistency} which is further attributed to~\cite{jin2015fast}.
	Let $\Cc_r$ denote the $r$th cluster of $X$, having center $q_r = q_r(X)$. We have $|\Cc_r| = n_r$. Let $x_i^T$ and $\xt_i^T$ be the $i$th row of $X$ and $\Xt$, respectively, and let
	\begin{align*}
	T_r := \{i \in \Cc_r:\; \dr(\xt_i,q_r) < c_r \delta_r\} = \{i \in \Cc_r:\; \dr(\xt_i,x_i) < c_r \delta_r\}
	\end{align*}
	using $x_i = q_r$ for all $i \in \Cc_r$ which holds by definition. Let $S_r = \Cc_r \setminus T_r$. Then, 
	\begin{align}\label{eq:temp:49985}
	|S_r|  c_r^2 \delta_r^2 \;\le\; \sum_{i\, \in\, S_r} \dr(\xt_i,x_i)^2 \le \eps^2 \implies \frac{|S_r|}{|\Cc_r|} \le \frac{c_r^{-2} \eps^2}{ n_r \delta_r^2}<1.
	\end{align}
	where we have used assumption~(b). It follows that $S_r$ is a proper subset of $\Cc_r$, that is, $T_r$ is nonempty for all $r \in [\kk]$.
	
	
	Next, we argue that  if two elements belong to different $T_r, r \in [\kk]$, they have different labels according to $\Xt$. That is, $i \in T_r$, $j \in T_\ell$ for $r \neq \ell$ implies $\xt_i \neq \xt_j$. Assume otherwise, that is, $\xt_i = \xt_j$. Then, by triangle inequality and $c_r + c_\ell \le 1$,
	\begin{align*}
	\dr(q_k, q_\ell) \le \dr(q_k,\xt_i) + \dr(q_\ell,\xt_j) < c_r \delta_r + c_\ell \delta_\ell \le \max\{\delta_r, \delta_\ell\}
	\end{align*}
	contradicting~\eqref{eq:center:sep:alt}. This shows that $\Xt$ has at least $\kk$ labels, since all $T_r$ are nonempty, hence exactly $\kk$ labels, since $\Xt \in \kmmt$ by assumption.
	
	Finally, we argue that if two elements belong to the same $T_r$, they have the same label according to $\Xt$. This immediately follows from the previous step since otherwise there will be at least $\kk+1$ labels. Thus, we have shown that, for all $r \in [\kk]$, the labels in each $T_r$ are in the same cluster according to both $X$ and $\Xt$, that is, they are correctly classified. The misclassification rate over cluster $\Cc_r$ is then $\le |S_r|/|\Cc_r|$ which establishes the result in view of~\eqref{eq:temp:49985}.
\end{proof}

\subsection{Proofs of Section~\ref{sec:consist:res}}\label{sec:proof:consist:res}

\begin{proof}[Proof of Theorem~\ref{thm:scone}]
	Going through the three-step plan of analysis in Section~\ref{sec:analysis:steps}, we observe that~\eqref{eq:gen:concent} holds for $\Are$ by Theorem~\ref{thm:concent:nonsym}, and~\eqref{eq:Z:dev:bound} holds by Lemma~\ref{lem:DK:Z:dev}. We only need to verify conditions of Corollary~\ref{cor:kmeans:misclass}, so that $\kappa$-approximate \kmeans operator $\Pc_\kappa$ satisfies  bound~\eqref{eq:lqc:kmeans} of the \kmeans step. As in the proof of Theorem~\ref{thm:sc:prototype}, $\Xs =  \Zb_1 O \in \rmat(n_1,\kk)$, where $\Zb_1 \in \ort{n_1}{\kk_1}$ and $O := \Usi Q \in \ort{\kk_1}{\kk}$. Clearly, $\Xs$ has exactly $\kk$ distinct rows (recalling $\kk = \kk_1$). Furthermore, using the calculation in the proof of Theorem~\ref{thm:sc:prototype},
	\begin{align*}
	n_{1t}\,\delta_t^2  = n_{1t} \min_{s:\; s\neq t} (n_{1t}^{-1} + n_{1s}^{-1}) =  \min_{s:\; s\neq t} (1 + \frac{n_{1t}}{n_{1s}}) \ge 1.
	\end{align*}
	Recalling that $\eps^2 = C_2^2 \kk \dg/\sigma_\kk^2$, as long as
	\begin{align*}
	4 (1+\kappa)^2 \eps^2 = 4 C_2^2 (1+\kappa)^2 \,\kk \dg /\sigma_\kk^2 < 1 \le n_{1t}\,\delta_t^2
	\end{align*}
	condition~(b) of Corollary~\ref{cor:kmeans:misclass} holds and $\Pc_\kappa$ satisfies~\eqref{eq:lqc:kmeans} with $c = 4(1+\kappa)^2$ as in~\eqref{eq:Misinf:Misb:approx:kmeans}. The rest of the proof follows as in Theorem~\ref{thm:sc:prototype}.
\end{proof}

\begin{proof}[Proof of Lemma~\ref{lem:RR:dev}]
	Throughout the proof, let $\mnorm{\cdot} = \opnorm{\cdot}$ be the operator norm.
	Recall that $P = \ex A$ is the mean matrix itself, and let $\Delt := \Are - P$. 
	By Weyl's theorem on the perturbation of singular values, $|\sigma_i(\Are) - \sigma_i(P)| \le \mnorm{\Delt}$ for all $i$. Since $\sigma_{\kk+1}(P)  = 0$ (see~\eqref{eq:P:svd}), we have $\sigma_{\kk+1}(\Are) \le \mnorm{\Delt}$, hence
	\begin{align*}
	\mnorm{\Arek - P} &\le \mnorm{\Arek - \Are} + \mnorm{\Delt} &&(\text{triangle inequality})\\
	&= \sigma_{\kk+1}(\Are) + \mnorm{\Delt} &&(\text{by~\eqref{eq:approx:err:k:reduced:SVD}})\\
	&\le 2\mnorm{\Delt} &&(\text{Weyl's theorem}).
	\end{align*}
	Thus, in terms of the operator norm, we lose at most a constant in going from $\Are$ to $\Arek$. However, we gain a lot in Frobenious norm deviation. Since $\Are$ is full-rank in general, the best bound on $\Delt$ based on its operator norm is $\mnorm{\Delt}_F \le \sqrt{n_\wedge} \,\mnorm{\Delt}$ where $n_\wedge = \min\{n_1,n_2\}$. On the other hand, since $\Arek - P$ is of rank $ \le 2\kk$, we get
	\begin{align*}
	\mnorm{\Arek - P}_F \le \sqrt{2\kk}\mnorm{\Arek - P} \le 2 \sqrt{2\kk} \mnorm{\Delt}.
	\end{align*}  
	Combining with~\eqref{eq:gen:concent}, that is, $\mnorm{\Delt} \le C \sqrt{\dg}$, we have the result.
\end{proof}

\begin{proof}[Proof of Theorem~\ref{thm:scrr}]
	We only need to calculate $\delta(P) = \delinf(P)$ the minimum center separation of $P$ viewed as an element of $\kmm(n_1,n_2,\kk_1)$. Recall that
	\begin{align*}
	P = \Zb_1 \Bb\Zb_2^T = Z_1 N_1^{-1/2} (\Nb_1^{1/2} \Psi \Nb_1^{1/2}) \Zb_2^T = n_1^{-1/2} Z_1 \Psi \Nb_1^{1/2} \Zb_2^T.
	\end{align*}
	Let $e_s$ be the $s$th standard basis vector of $\reals^{\kk_1}$. Unique rows of $P$ are  $q_s^T := n_1^{-1/2} e_s^T (\Psi \Nb_2^{1/2}) \Zb_2^T$ for $s \in [\kk_1]$.  We have
	\begin{align*}
	\norm{q_s - q_t}_2^2 &= n_1^{-1}\norm{\Zb_2 \Nb_2^{1/2} \Psi^T (e_s - e_t) }_2^2 \\
	&= n_1^{-1} \norm{\Nb_2^{1/2} \Psi^T (e_s - e_t) }_2^2 
	= n_1^{-1}\sum_{\ell=1}^{\kk_2} \pi_{2\ell} (\Psi_{s \ell} - \Psi_{t \ell})^2.
	\end{align*}
	It follows that
	$	\delta^2(P) = \min_{t \neq s }	\norm{q_s - q_t}_2^2 = n_1^{-1} \Psinf[1]^2$.
	We apply Corollary~\ref{cor:kmeans:misclass}, with $\Xs = P$ and $\Xh = \Arek$, taking $\eps^2 = 8 C^2 \kk \dg$ according to Lemma~\ref{lem:RR:dev}. Condition~(b) of the corollary holds if 
	\begin{align*}
	32 C^2 (1+\kappa)^2\kk \dg = 4(1+\kappa)^2 \eps^2 < n_{1t} \, \delta_{t}^2(P) = \pi_{1t} \min_{s :\; s\neq t} \sum_{\ell=1}^{\kk_2} \pi_{2\ell} (\Psi_{s \ell} - \Psi_{t \ell})^2 
	\end{align*}
	for all $t \in [\kk_1]$, which is satisfied under assumption~\eqref{eq:scrr:asump}. Corollary~\ref{cor:kmeans:misclass}, and specifically~\eqref{eq:Misinf:Misb:approx:kmeans} gives the desired bound on misclassification rate $\le 4 (1+\kappa)^2 \eps^2 / (n_1 \delta^2(P))$.
\end{proof}

\begin{proof}[Proof of Theorem~\ref{thm:scerr}]
	Recall that $\Arek = \Zh_1 \Sigh \Zh_2^T$ is the $k$-truncated SVD of $\Are$. Let $X^{(1)} =  \Zh_1 \Sigh$ and $X^{(2)} = \Arek$, and let $(x^{(1)}_i)^T$ and $(x^{(2)}_i)^{T}$ be their $i$th rows, respectively. Then, 
	\begin{align*}
	\norm{x^{(i)}_2 - x^{(j)}_2} = \norm{\Zh_2 (x^{(i)}_1 - x^{(j)}_1)}_2 = \norm{x^{(i)}_1 - x^{(j)}_1}_2, \quad \forall i\neq j,
	\end{align*}
	using $\Zh_2 \in \ort{n_2}{\kk}$ and~\eqref{eq:ort:isom:1}. Isometry-invariance of $\kalg$ implies $\Misb(\kalg(\Zh_1 \Sigh), \kalg(\Arek)) = 0$. Since $\Misb$ is a pseudo-metric on \kmeans matrices, using the triangle inequality, we get
	\begin{align*}
	\Misb\big(\kalg(\Zh_1 \Sigh), P\big) &\le \Misb\big(\kalg(\Zh_1 \Sigh), \kalg(\Arek)\big) 
	+ \Misb\big(\kalg(\Arek), P\big) \\
	&= \Misb\big(\kalg(\Arek), P\big).
	\end{align*}
	(In fact, using the triangle inequality in the other direction, we conclude that the two sides are equal.) The result now follows from Theorem~\ref{thm:scrr}.
\end{proof}

\begin{proof}[Proof of Corollary~\ref{cor:mean:param}]
	We have
	\begin{align*}
	\norm{\Lambda_{s*} - \Lambda_{t*}}^2 
	= \sum_{\ell=1}^{\kk_2} n_{2\ell}^2 (B_{s\ell} - B_{t\ell})^2 
	= \sum_{\ell=1}^{\kk_2} \frac{n_{2\ell}^2}{n_1 n_2} (\Psi_{s\ell} - \Psi_{t\ell})^2 
	= \alpha \sum_{\ell=1}^{\kk_2} \pi_{2\ell}^2 (\Psi_{s\ell} - \Psi_{t\ell})^2 
	\end{align*}
	where we have used $\alpha := n_2/n_1$. Recall from~\eqref{eq:Laminf:def} that $\Laminf^2 :=  \min_{t \neq s} \norm{\Lambda_{s*} - \Lambda_{t*}}^2$. 
	It follows that $\alpha^{-1}\Laminf^2 \le  \Psinf[1]^2$, using $\pi_{2,\ell} \le 1$. We also have $\Psinft[1]^2 \ge \pinf{1} \Psinf[1]^2 \ge \pinf{1} \alpha^{-1}\Laminf^2$. Recalling the definition of $a$ from~\eqref{eq:gen:concent}, and using $\Psi_{s \ell} = (\sqrt{n_1 n_2} / n_{2\ell}) \Lambda_{s\ell}$, we have 
	\begin{align*}
	\dg = \sqrt{\frac{n_2}{n_1}} \infnorm{\Psi} 
	\le  \sqrt{\frac{n_2}{n_1}} \frac{\sqrt{n_1 n_2}}{n_{2,\wedge}} \infnorm{\Lambda} 
	=  \frac1{\pinf{2}} \infnorm{\Lambda} \le \,\beta_2 \kk_2\,  \infnorm{\Lambda}.
	\end{align*}
	Hence, $\kk \dg \Psinf[1]^{-2} \le \kk  \,\beta_2 \kk_2\,  \infnorm{\Lambda} (\alpha^{-1}\Laminf^2)^{-1} = \beta_2 \kk \kk_2 \alpha \infnorm{\Lambda} \Laminf^{-2}$ which is the desired bound. We also note that
	\begin{align*}
	\kk \dg \Psinft[1]^{-2} \le  \kk  \dg \pinf{1}^{-1}\, \Psinf[1]^{-2} \le (\beta_1 \kk_1 )\, \kk \dg \Psinf[1]^{-2}
	\end{align*}
	which combined with the previous bound shows that the required condition~\eqref{eq:scrr:asump:Lambda} in the statement is enough to satisfy~\eqref{eq:scrr:asump}.
\end{proof}


\subsection{Proof of Theorem~\ref{thm:scone:unequal}}\label{sec:proof:thm:scone:unequal}
	As in the proof of Theorem~\ref{thm:scone}, we only need to verify conditions of Corollary~\ref{cor:kmeans:misclass}, with $\Xs =  \Zb_1 \Usi Q \in \rmat(n_1,\kk)$, so that $\kappa$-approximate \kmeans operator $\Pc_\kappa$ satisfies  bound~\eqref{eq:lqc:kmeans} of the \kmeans step.  Recall that $\Zb_1 \in \ort{n_1}{\kk_1}$, $\Usi \in \ort{\kk_1}{\kk}$ and $Q \in \ort{\kk}{\kk}$. Clearly, $\Xs$ has exactly $\kk_1$ distinct rows, $\usi_s^T Q, \,s \in [\kk_1]$. Furthermore, 
	\begin{align*}
	\delta^2 := \delinf^2(\Zb_1 \Usi Q) = \delinf^2(\Zb_1\Usi)
	&= \min_{(t,s):\, t \neq s} \norm{n_{1t}^{-1/2} \usi_{t} - n_{1s}^{-1/2} \usi_{s}}_2^2 
	\end{align*}
	where the second equality is by~\eqref{eq:ort:isom:1}. Letting $\norm{\cdot} = \norm{\cdot}_2$, and $x =\pi_{1t}^{-1/2} $ and $ y=\pi_{1s}^{-1/2}$ for simplicity, and $z = (x,-y) \in \reals^2$,  we have
	\begin{align*}
	n_1 \norm{n_{1t}^{-1/2} \usi_{t} - n_{1s}^{-1/2} \usi_{s}}^2 &= 
	\norm{\pi_{1t}^{-1/2} \usi_{t} - \pi_{1s}^{-1/2} \usi_{s}}^2  \\
	&= \norm{ \begin{pmatrix}
		u_t & u_s
		\end{pmatrix}
		\begin{pmatrix}
		x \\ -y
		\end{pmatrix}
	}^2 \\ &=z^T (\Usi \Usi^T)_{\Ic}\, z, &\text{(where $\Ic = \{t,s\}$),}\\
	&= \norm{z}^2 -  z^T (I_{\kk_1} - \Usi \Usi^T)_{\Ic} \, z^T.
	\end{align*}
	Since $I_{\kk_1} - \Usi \Usi^T$ is positive semidefinite, using~\eqref{eq:U:incoh}, we have
	\begin{align*}
	0 \le z^T (I_{\kk_1} - \Usi \Usi^T)_{\Ic} \, z^T \le (1-\rho_1) \norm{z}^2,
	\end{align*}  
	hence, $n_1 \norm{n_{1t}^{-1/2} \usi_{t} - n_{1s}^{-1/2} \usi_{s}}^2 \ge \rho_1 \norm{z}^2 = \rho_1 (\pi_{1t}^{-1} + \pi_{1s}^{-1})$. 	It follows by~\eqref{asu:clust:prop},
	\begin{align*}
	(n_1 \delta^2 )^{-1} \le \frac1{\rho_1} \max_{t\neq s} \;(\pi_{1t}^{-1} + \pi_{1s}^{-1})^{-1} 
	\le  \frac1{\rho_1} \Big(\frac{\beta_1}{2 \kk_1}\Big).
	\end{align*}
	We also have
	\begin{align*}
	n_{1t} \,\delta_t^2 = n_{1t} \,\min_{s:\, s\neq t} \norm{n_{1t}^{-1/2} \usi_{t} - n_{1s}^{-1/2} \usi_{s}}_2^2 
	\; \ge \; \rho_1 \big(1 + \frac{\pi_{1t}}{\pi_{1s}}\big) \;\ge\; \rho_1.
	\end{align*}
	We obtain, with $\eps^2 = C_2^2 \kk \dg/\sigma^2_{\kk}$ and $O= \Usi Q$,
	\begin{align*}
	\Misb(\kalg(\Zh_1),\Zb_1) = \Misb(\kalg(\Zh_1),\Zb_1 O)
	\; \lesssim \;\frac{\eps^2}{ n_1 \delta^2} \;\lesssim\; \frac{\beta_1}{\rho_1}\frac{\kk}{\kk_1}  \frac{\dg}{\sigma_\kk^2}   
	\end{align*}
	as long as $4 (1+\kappa)^2 ( n_{1t}\,\delta_t^2)^{-1} \eps^2 = 4 C_2^2 (1+\kappa)^2 \, \rho_1^{-1}\,\kk \dg/\sigma_\kk^2 < 1 $, to guarantee that condition~(b) of Corollary~\ref{cor:kmeans:misclass} holds, so that $\Pc_\kappa$ satisfies~\eqref{eq:lqc:kmeans} with $c = 4(1+\kappa)^2$ as in~\eqref{eq:Misinf:Misb:approx:kmeans}. The proof is complete.

\subsection{Proof of Theorem~\ref{thm:scerr:subg:noise}}\label{sec:proof:subg}
The following concentration result for sub-Gaussian random matrices is well-known~\cite{vershynin2018high}:

\begin{lem}\label{lem:opnorm:sg}
	Let $A \in \reals^{n_1 \times n_2}$ have independent sub-Gaussian entries, and $\sigma = \max_{i,j} \|A_{ij}-\ex A_{ij}\|_{\psi_2}$. Then for any $t>0$, we have
	\begin{align}
	\opnorm{A- \ex A} \le C \,\sigma (\sqrt{n_1}+\sqrt{n_2}+t),
	\end{align}
	with probability at least $1-2\exp(-t^2)$.
\end{lem}

	Let $t =\sqrt{n_1}+\sqrt{n_2}$ in Lemma~\ref{lem:opnorm:sg}, then $\opnorm{A-\ex A} \le 2 C \sigma (\sqrt{n_1}+\sqrt{n_2}) $
	with probability at least $1-2e^{-(n_1+n_2)}$. We thus have the concentration bound~\eqref{eq:gen:concent} with $\sqrt{\dg} = 2\sigma (\sqrt{n_1} + \sqrt{n_2})$. The proof of Theorem~\ref{thm:scrr} goes through and the result follows by replacing $\dg$ with the upper bound $\dg \le 4 \sigma^2 (n_1+n_2)$, and replacing $\Psinf[1]^2$ and $\Psinft[1]^2$ with $n_1 n_2 \Binf[1]^2$ and $n_1 n_2 \Binft[1]^2$, respectively.
	%

\subsection{Proofs of Section~\ref{sec:inhom:graphs}}\label{sec:proofs:inhom:graphs}

\begin{proof}[Proof of Proposition~\ref{prop:irgm:concent}]
	We have with the given probability,
	\begin{align*}
	\opnorm{\Are - \Pt } &\le 	\opnorm{\Are - P } + 	\opnorm{P - \Pt } \\
	&\le C \sqrt{\dg} + \frac{1}{\sqrt{n_1 n_2}} \opnorm{\Pn - \Pnt} \\
	&\le C \sqrt{\dg} + \frac{1}{\sqrt{n_1 n_2}} \fnorm{\Pn - \Pnt} \le (C_0 + C) \sqrt{\dg}
	\end{align*}
	where we have used Theorem~\ref{thm:concent:nonsym} to bound the first term using $d' = d = \sqrt{n_2/n_1} \infnorm{P^0} = n_2 \infnorm{P}$. The last inequality uses assumption~\eqref{asu:Pn:dev}. The proof is complete.
\end{proof}

\begin{proof}[Proof of Proposition~\ref{prop:rho:dev}]
	
	Letting $f(x,y) := [\rhoz(x,y)-\rhozt(x,y)]^2$,
	\begin{align*}
	\|\Pn-\Pnt\|_F^2
	=\sum_{i=1}^n \sum_{j=1}^n f(X_i,Y_j)
	=\sum_{j\in\mathbb Z_n} \sum_{i\in \mathbb Z_n} f(X_i, Y_{i+j}),
	\end{align*}
	where $\ints_n = \ints / n \ints$ is $\{1,\dots,n\}$ viewed as a cyclic group of order $n$.
	We have 
	\begin{align*}
	\ex[f(X_1,Y_1)^2]=\|f\|_{L^2}^2=\|\rhoz-\rhozt\|_{L^4}^4=o(\dg^2)
	\end{align*}
	and $\ex[f(X_1,Y_1)] = \norm{f}_{L^1} \le  \norm{f}_{L^2} =o(\dg)$. Note that $\{f(X_i, Y_j):i,j\in \mathbb Z_n\}$ are not independent, however, for each $j\in\mathbb Z_n$, $\{f(X_i, Y_{i+j}):i\in \mathbb Z_n\}$ are i.i.d.. Let $Z_j :=  \sum_{i\in \mathbb Z_n} f(X_i, Y_{i+j}) $. By independence, 
	\begin{align*}
	\var(Z_j) = n \var(f(X_1,Y_{1})) \le n \norm{f}_{L^2}^2.
	\end{align*}
	Since $\ex Z_j = n \norm{f}_{L^1}$, by the Chebyshev inequality, for any $j \in \ints_n$,
	\begin{align*}
	\pr(Z_j \ge n\dg) &= \pr\big(Z_j - \ex Z_j \ge n\dg - n\norm{f}_{L^1} \big) \\
	&\le \frac{n \norm{f}_{L^2}^2}{(na - n\norm{f}_{L^1})^2} 	= \frac1n \frac{o(\dg^2)}{(\dg - o(\dg))^2} = o\Big( \frac 1n \Big).
	\end{align*} 
	It follows that $\pr\big( \fnorm{\Pn-\Pnt}^2 \ge  n^2 \dg \big) \le \sum_{j \in \ints_n} \pr\big( Z_j \ge n\dg \big) \le n \,o\big( n^{-1} \big) = o(1).$
\end{proof}


\section{Alternative algorithm for the \kmeans step}\label{app:kmeans:replace}

In this appendix, we present a simple general algorithm that can be used in the \kmeans step, replacing the $\kappa$-approximate \kmeans solver used throughout the text. The algorithm is based on the ideas in~\cite{gao2017achieving} and~\cite{yun2014community}, and the version that we present here acheives the misclassification bound  $\eps^2/(n\delta^2)$ needed in Step~3 of the analysis (Section~\ref{sec:analysis:sketch}) without necessarily optimizing the \kmeans objective function. We present the results using the terminology of the \kmeans matrices (with rows in $\reals^d$) introduced in Section~\ref{sec:kmeans:step}, although the algorithm and the resulting bound work for data points in any metric space.

Let $X \in \kmmt$ be a \kmeans matrix and let us denote its centers, i.e. distinct rows, as $\{q_r(X), r \in [\kk]\}$. As in Definition~\ref{dfn:center:sep}, we write $\delta_r(X)$ and $n_r(X)$ for the $r$th cluster center separation and size, respectively, and $\delinf(X) = \min_r \delta_r(X)$ and $\ninf= \min_r n_r(X)$.
Assume that we have an estimate $\Xh \in \reals^{n \times d}$ of $X$, and let us write $d(i,j) := d(\xh_i,\xh_j), \, i,j \in [n]$ for the pairwise distances between the rows of $\Xh$. 

Algorithm~\ref{alg:kmeans:replace} which is a variant of the one presented in~\cite{gao2017achieving}, takes these pairwise distances and outputs cluster estimates $\Ch_1,\dots,\Ch_k \subset [n]$, after $k$ recursive passes through the data. A somewhat more sophisticated version of this algorithm appears in~\cite{yun2014community}, where one also repeats the process for $i=1,\dots,\log n$ and radii $R_i = i R_1 $ in an outer loop, producing clusters $\Ch_{r}^{(i)},r \in [k]$; one then picks, among these $\log n$ possible clusterings, the one that minimizes the \kmeans objective. The variant in~\cite{yun2014community} also leaves no unlabeled nodes by assigning the unlabeled to the cluster whose estimated center is closest. In the rest of this section, we will focus on the simple version presented in Algorithm~\ref{alg:kmeans:replace} as this is enough to establish our desired bound. The following theorem provides the necessary guarantee:

%
%
%
%

\begin{algorithm}[t]
	\caption{\kmeans replacement}
	\label{alg:kmeans:replace}
	\begin{algorithmic}[1]
		\Require Pairwise distance $d(i,j)$, $i,j \in [n]$ and radius $\rho$.
		\State $S \gets [n]$
		\For{$r = 1,\dots, \kk$}
			\State For every $i \in S$, let $ B_d(i;\rho) := \{j \in S:\; d(i,j) \le \rho\}$.
			\State Pick $i_0 \in S$ that maximizes $i \mapsto |B_d(i;\rho) |$.
			\State Let $\Ch_r = B_d(i_0;\rho)$.
			\State $S \gets S \setminus \Ch_r$.
		\EndFor
		\Ensure Return clusters $\Ch_r: 1,\dots,\kk$ and output remaining $S$ as unlabeled.
	\end{algorithmic}
\end{algorithm}

\begin{thm}\label{thm:kmeans:replace}
	Consider the cluster model above and let  $n_r = n_r(X)$, $\ninf = \ninf(X)$ and $\delinf = \delinf(X)$. Assume that we have approximate data $\xh_1,\dots,\xh_n$ such that $\sum_{i=1}^n d(x_i,\xh_i)^2 \le \eps^2$. In addition, assume that for some $\gamma \in (0,1)$ and $\beta \ge 1$:
	\begin{itemize}
		\item[(i)] $n_r \le \beta \ninf$ for all $r \in [k]$ \quad(Clusters are $\beta$-balanced.),
		\item[(ii)] $\displaystyle \frac{2\eps}{\sqrt{ \gamma \ninf }} < \frac{\delinf}3$ \quad  ($\eps^2$ small enough compared to $\ninf \delinf^2$.),
		\item[(iii)] $\xi\beta+ \gamma < 1-\gamma$ where $\xi := \gamma/(1-\gamma)$. (Gamma small enough relative to $\beta$.)
	\end{itemize}
	Let $M_n(\rho)$ be the (average) misclassification rate of  Algorithm~\ref{alg:kmeans:replace} with input radius $\rho$. Then,
	\begin{align*}
		M_n(\rho) \le \frac{8 \eps^2}{n \rho^2} ,\quad \forall \rho \in \Big[\frac{2\eps}{\sqrt{\gamma \ninf}}, \frac\delinf3\Big).
	\end{align*}
	
\end{thm}

Applying the algorithm with $\rho \ge \alpha \delinf$ for $\alpha < 1/3$ we obtain the misclassication bound $c_\alpha \,\eps^2/(n\delinf^2)$ where $c_\alpha = 8/\alpha^2$. Thus, Algorithm~\ref{alg:kmeans:replace} with a proper choice of the radius $\rho$ satisfies the desired bound~\eqref{eq:lqc:kmeans} of the \kmeans step.


\begin{proof}[Proof of Theorem~\ref{thm:kmeans:replace}] The proof follows the argument in~\cite{gao2017achieving}.
	 As in the proof of Proposition~\ref{prop:kmeans:misclass}, 	let $\Cc_r$ denote the $r$th cluster of $X$, having center $q_r = q_r(X)$. We have $|\Cc_r| = n_r$. Let $x_i$ and $\xh_i$ be the $i$th row of $X$ and $\Xh$, respectively, and let
	\begin{align*}
	T_r := \{i \in \Cc_r:\; \dr(\xh_i,q_r) < \rho/2\} = \{i \in \Cc_r:\; \dr(\xh_i,x_i) <  \rho/2\}
	\end{align*}
	using $x_i = q_r$ for all $i \in \Cc_r$ which holds by definition. $\{T_r\}$ are disjoint and clearly $T_r \subset \Cc_r$. Let $T := \biguplus_r T_r$, a disjoint union, and $T^c = [n] \setminus T$. We have
	\begin{align}\label{eq:Tc:upper:bound}
	|T^c|  \rho^2/4 \;\le\; \sum_{i\, \in\, T^c} \dr(\xh_i,x_i)^2 \le \eps^2
		 \implies  |T^c| \le 4\eps^2/\rho^2.
	\end{align} 
%
	As a consequence of assumption (ii) and our choice of $\rho$, we have $4\eps^2/(\ninf \rho^2) \le \gamma$, hence
	\begin{align}\label{eq:Tk:size:lower:bound}
	|T_r| = |\Cc_r| - |\Cc_r \setminus T_r| \;\ge\; |\Cc_r| - |T^c| 
	\;\ge\; \ninf\Big(1- \frac{4\eps^2}{\ninf \rho^2} \Big) \ge \ninf (1-\gamma)
	\end{align}
	for all $r \in [\kk]$.
	On the other hand, $|T^c| \le \gamma \ninf$. In particular, combining the two estimates
	\begin{align}\label{eq:Tc:Tr:bound}
		|T^c| \le \xi \,|T_r|, \quad \forall r \in [\kk]
	\end{align}
	where $\xi = \gamma/(1-\gamma)$. These size estimates will be used frequently in the course of the proof. 
	
	\medskip
	Recall that $d(i,j) := d(\xh_i,\xh_j), \, i,j \in [n]$, the collection of pairwise distances between the data points $\xh_1,\dots,\xh_n$. Thus, with some abuse of notation, $(i,j) \mapsto d(i,j)$ defines a pseudo-metric on $[n]$ (and a proper metric if $\{\xh_i\}$ are distinct).  For any two subsets $A,B \subset [n]$ we write $d(A,B) = \inf\{d(i,j): i \in A,\, j \in B\}$. For any $i \in [n]$, let $d(i,A) = d(\{i\},A)$.
	
	\medskip
	We say that node $i_0$ is near $T$ if $d(i_0, T) \le \rho$, i.e., $i_0$ belongs to the $\rho$-enlargement of $T$. Similarly, we say that $i_0$ is near $T_r$ if $d(i_0,T_r) \le \rho$ and far from $T_r$ otherwise. Note that $i_0$ can be near at most one of $T_r, r \in [\kk]$. This is since $d(T_r,T_\ell) \ge \delinf -\rho$ for $r\neq \ell$, and we are assuming $\delinf > 3\rho$.
	In fact, \emph{$i_0$ is near $T$ iff $i_0$ is near exactly one of $T_r, r \in [\kk]$}.

	To understand Algorithm~\ref{alg:kmeans:replace}, let us assume that we are at some iteration of the algorithm and we are picking the center $i_0$ and the corresponding cluster $\Ch := \{j :\; d(j,i_0) \le \rho \}$. One of the following happens: 
	\begin{itemize}
		\item[(a)] We pick the new center $i_0 \in T_r$ for some $r$, in which case $\Ch$ will include the entire $T_r$, none of $T_\ell, \ell \neq r$, and perhaps some of $T^c$. That is, $\Ch \supset T_r$ and $\Ch \cap T_\ell = \emptyset$ for $\ell \neq r$.
		
		\item[(b)] We pick $i_0$ \emph{near} $T_r$ for some $r$. In this case,  $|\Ch| \ge |T_r|$, otherwise any member of $T_r$ would have created a bigger cluster by part~(a) above.  Now, $\Ch$ cannot contain any of $T_\ell, \ell\neq r$, because $i_0$ is far from those if it is near $T_r$. Hence, $\Ch \subset T_r \cup T^c$. Since $|T^c| \le \xi |T_r|$ by~\eqref{eq:Tc:Tr:bound}, and $|\Ch| \ge |T_r|$, we have  $|\Ch \cap T_r| \ge (1-\xi) |T_r|$. That is, $\Ch$ contains a large fraction of $T_r$.
		
		
	\end{itemize}
	If either of the two cases above happen, we say that $T_r$ is depleted, otherwise it is intact. If $T_r$ is depleted, it will not be revisited in future iterations, as long as other intact $T_\ell, \ell \neq r$ exist. To see this, first note that $|T_r \cap \Ch^c| \le \xi |T_r| \le \xi \beta \ninf$, using assumption~(i). Taking $i_0$ on or near $T_r$ in a future iteration will give us a cluster of size at most $(\xi \beta + \gamma) \ninf < (1-\gamma) \ninf$ (by assumption~(iii)) which is less that $|T_\ell|$ for an intact cluster. 
	
	
	\medskip
	To simplify notation, if either of (a) and (b) happen, i.e., we pick cluster center $i_0$ near $T_r$ for some $r$, we name the corresponding cluster $\Ch_r$. This is to avoid carrying around a permutation of cluster labels different than the original one, and is valid since each $T_r$ is visited at most once by the above argument. (In fact, each is visited exactly once, as we argue below.)
	%
	That last possibility is
	\begin{itemize}
		\item[(c)] We pick $i_0$ \emph{far} from any $T_r$, that is $d(i_0,T) > \rho$. This gives $\Ch \subset T^c$, hence $|\Ch| \le |T^c| \le \gamma \ninf < (1-\gamma) \ninf \le |T_\ell|$ for any intact $T_\ell$. Thus as along as there are intact $T_\ell$, this case does not happen.
	\end{itemize}

	The above argument gives the following picture of the evolution of the algorithm: At each step $t=1,\dots,\kk$, we pick $i_0$ near $T_r$ for some previously unvisited $r$, making it depleted, creating estimated cluster $\Ch_r$ and proceeding to the next iteration. After the $\kk$-th iteration all $T_\ell, \ell \in [\kk]$ will be depleted. We have $|\Ch_r| \ge |T_r|$, and $\Ch_r \subset T_r \cap T^c$ for all $r \in [\kk]$. 
	
	\medskip
	By construction $\{\Ch_\ell\}$ are disjoint. Let $\Ch := \biguplus_{\ell \in [\kk]} \Ch_\ell$, and note that $|\Ch| \ge |T|$ hence $|\Ch^c| \le |T^c|$. Since $\Ch_\ell \cap T_r = \emptyset $ for $\ell \neq r$, we have $T_r \subset \bigcap_{\ell \neq r} \Ch_\ell^c$, hence $T_r \cap \Ch_r^c \subset \Ch^c$. 
	All the misclassified or unclassified nodes produced by the algorithm are contained in $[\bigcup_r (T_r \cap \Ch_r^c)] \cup T^c$ which itself is contained in $\Ch^c \cup T^c$. Hence, the misclassification rate is bounded above by
	\begin{align*}
	\frac{1}n |\Ch^c \cup T^c| \le 	\frac{1}n (|\Ch^c| + |T^c|) \le \frac{2}n| T^c| \le \frac{8\eps^2}{n \rho^2}.
	\end{align*}
	where we have used~\eqref{eq:Tc:upper:bound}. The proof is complete.
\end{proof}

	\begin{rem}
	The last part of the argument can be made more transparent as follows: Each $\Ch_r$ consists of two disjoint part, $\Ch_r \cap T_r$ and $\Ch_r \cap T^c$. We have $|\Ch_r \cap T^c  | \ge |T_r \setminus \Ch_r|$ (equivalent to $|\Ch_r| \ge |T_r|$). Then,
	\begin{align*}
	\sum_{k} |T_r \setminus \Ch_r| \le \sum_r |\Ch_r \cap T^c  | = |\Ch \cap T^c| \le |T^c|
	\end{align*}
	and total misclassifications are bounded by $\sum_{k} |T_r \setminus \Ch_r| + |T^c|$. 
\end{rem}


\section{Auxiliary lemmas}\label{app:aux}

\begin{lem}\label{lem:proj:align}
	Let $Z,Y \in \ort{n}k$ and let $\proj_{Z}=ZZ^T$ and $\proj_{Y}=YY^T$ be the corresponding projection operators. We have
	\begin{align*}
		\min_{Q \,\in\, \ort{k}k} \fnorm{Z - YQ} \;\le\; \fnorm{\proj_Z - \proj_Y}.
	\end{align*} 
\end{lem}
\begin{proof}
	We first note that $\fnorm{\proj_Z}^2 = \tr(\proj_Z^2) = \tr(\proj_Z) = k$ (since projections are idempotent), and $\fnorm{Z}^2 = \tr(Z^T Z) = \tr(\proj_Z) = k$. Let $Z^T Y = U \Sigma V^T$ be the SVD of $Z^T Y$ where $U,V \in \ort{k}k$ and $\Sigma = \diag(\sigma_1,\dots,\sigma_k) \succeq 0$.  Then, using the change of variable $O = V^T Q U$,
	\begin{align*}
		\frac12 \min_Q \fnorm{Z - YQ}^2 &= k - \max_Q \tr(Z^T Y Q) \\
		&= k - \max_{O \;\in\; \ort{k}k} \tr(\Sigma O) = k - \nucnorm{\Sigma},
	\end{align*}
	where $\nucnorm{\Sigma} = \sum_i \sigma_i$ is the nuclear norm of $\Sigma$.
	To see the last equality, we note that since $O$ is orthogonal, we have 
	$|O_{ii}| \le 1$ for all $i$, hence $\max_{O} \tr(\Sigma O) \le \max_{ \forall i, \; |O_{ii}| \le1} \sum_{i} \sigma_i O_{ii} = \sum_{i} \sigma_i$ by the duality of $\ell_1$ and $\ell_\infty$ norms and $\sigma_i \ge 0$. The equality is achieved by  $O = I_k$.
	 On the other hand
	\begin{align*}
		\frac12\fnorm{\proj_Z - \proj_Y}^2 = k -  \tr(\proj_Z\proj_Y) = k -  \fnorm{Z^T Y}^2 = k - \fnorm{\Sigma}^2. 
	\end{align*}
	Since $\opnorm{\Sigma} = \opnorm{Z^T Y} = \opnorm{Z} \opnorm{Y} \le 1$, we have $\sigma_i \le 1$ for all $i$. It follows that $\fnorm{\Sigma}^2 \le \nucnorm{\Sigma}$ completing the proof.
\end{proof}

 \pagebreak
 
 \end{document}